\documentclass[a4paper]{amsart}
\usepackage{amsmath,amsthm,amssymb}
\usepackage[T1]{fontenc}
\usepackage{url}
\usepackage{amsmath,amsthm,amssymb}
\usepackage{verbatim}
\usepackage{mathrsfs}
\usepackage{graphics}
\usepackage{graphicx}
\usepackage{subfigure}

\newtheorem{theorem}{Theorem}[section]
\newtheorem{proposition}[theorem]{Proposition}
\newtheorem{lemma}[theorem]{Lemma}
\newtheorem{corollary}[theorem]{Corollary}

\theoremstyle{remark}
\newtheorem{remark}[theorem]{Remark}
\newtheorem{definition}{Definition}

\newtheorem*{notation}{Notation}

\numberwithin{equation}{section}

\newcommand{\vep}{\varepsilon}
\newcommand{\R}{{\mathbb{R}}}

\newcommand{\Q}{{\mathbb{Q}}}
\newcommand{\C}{{\mathbb{C}}}
\newcommand{\Z}{{\mathbb{Z}}}
\newcommand{\N}{{\mathbb{N}}}

\newcommand{\xbm}{(X,\mathcal{B},\mu)}

\newcommand{\Int}{\operatorname{Int}}

\newcommand{\supp}{\operatorname{supp}}

\newcommand{\hol}{\operatorname{hol}}

\let\oldmarginpar\marginpar
\renewcommand\marginpar[1]{\-\oldmarginpar[\raggedleft\footnotesize #1]%
{\raggedright\footnotesize #1}}

\begin{document}
\title
{Recurrence and non-ergodicity in generalized wind-tree models}
\author[K. Fr\k{a}czek \and P. Hubert]{Krzysztof Fr\k{a}czek \and Pascal Hubert}

\address{Faculty of Mathematics and Computer Science, Nicolaus
Copernicus University, ul. Chopina 12/18, 87-100 Toru\'n, Poland}
\email{fraczek@mat.umk.pl}
\address{Institut de Math\'ematique de Marseille \\ Université d$ '$Aix-Marseille\\
Centre de Mathématiques et Informatique (CMI) \\
39, rue F. Joliot Curie \\
F-13453 Marseille Cedex 13, France} \email{pascal.hubert@univ-amu.fr}
\date{\today}

  \keywords{}
\thanks{Research partially supported by the National Science Centre (Poland) grant 2014/13/B/ST1/03153 and by the projet ANR blanc GeoDyM (France)}
\maketitle

\begin{abstract}
In this paper, we consider generalized wind-tree models and $\Z^d$-covers over compact translation surfaces. Under suitable hypothesis, we prove recurrence of the linear flow in a generic direction and non-ergodicity of Lebesgue measure.
\end{abstract}

\section{Introduction}

The wind-tree model is a billiard in the plane introduced by P.~Ehrenfest and T.~Ehrenfest in 1912 (\cite{EhEh}). Its periodic version is defined as follows. A point moves in the plane $\R^2$ and bounces elastically off rectangular scatterers following the usual law of reflection. The scatterers are translates of the rectangle $[0,a]\times[0,b]$ where $0<a<1$ and $0<b<1$, one centered at each point of $\Z^2$. The complement of obstacles in the plane is the wind-tree model, it is an infinite billiard table.

In this work, we study periodic polygonal billiards which are natural generalizations of the classical wind-tree model. Recently, results describing the
dynamics of the wind-tree model were obtained (\cite{AH}, \cite{Co-Gu}, \cite{Del}, \cite{DHL}, \cite{Del-Zo}, \cite{Fr-Ulc:nonerg}, \cite{Hu-Le-Tr}) based
on deep ideas in Teichm\"uller dynamics. For instance, the Kontsevich-Zorich cocycle and the Oseledets genericity result by Chaika and Eskin (\cite{Es-Ch})
are key tools in these studies.

In geometric terms, the billiard flow induces a linear flow on a doubly periodic surface $X_\infty$ which is a $\Z^2$-cover of a compact translation
surface $X$. It is known that, for every parameters $(a,b)$ and almost every direction, the linear flow is recurrent (\cite{AH}) and non ergodic
(\cite{Fr-Ulc:nonerg}) on $X_\infty$. Symmetries of the obstacles imply that $X$ is a $(\Z/2\Z)^2$-cover of a genus 2 surface. Most results rely on this
remark: the Hodge bundle splits into several invariant parts of small dimension. Moreover $GL(2,\R)$ orbit closures are classified in genus 2 by the work
of McMullen (\cite{McM}).

The present work is an attempt to tackle these questions for models with fewer symmetries  and for general $\Z^d$-covers of compact translation surfaces.
For more on the dynamics of the linear flow on non-compact translation surfaces, see the following list of references (\cite{Fr-Ul:erg_bil}, \cite{Ho1},
\cite{Ho2}, \cite{Ho3}, \cite{Ho-We}, \cite{Hu-We1}, \cite{Hu-We2}). Here, we state results about recurrence and non-ergodicity. A $\Z^d$-cover
of a translation surface $(X, \omega)$ is defined by $d$ independent cycles $\gamma=([\gamma_1], \ldots, [\gamma_d])$ in the absolute homology. The resulting translation
surface we denote by $(\widetilde{M}_\gamma,\widetilde{\omega}_\gamma)$. A necessary
condition for recurrence for $(\widetilde{X}_\gamma,\widetilde{\omega}_\gamma)$ is that
\[\int\limits_{\substack{ \gamma_1 }}\omega = 0, \ldots, \int\limits_{\substack{ \gamma_d }}\omega =0,\]
which means that $([\gamma_1], \ldots, [\gamma_d])$ belong to $\ker(\hol)$, the kernel of holonomy.

For the simplest class of square-tiled translation surfaces we give an effective sufficient condition for its holonomy free $\Z^d$-covers to be non-ergodic.
\begin{theorem} \label{thm:nonergo-easy}
Let $(X,\omega)$ be a compact square-tiled translation surface of genus $g>1$ such that the Kontsevich-Zorich cocycle over $SL(2,\R)\cdot \omega$ has $1<p\leq g$ positive Lyapunov exponents.
Let $d\geq2g-1-p$. Then for every $\gamma\in (\ker(\hol)\cap H_1(X,\Z))^d$ and a.e.\ direction $\theta\in S^1$ the directional flow on the $\Z^d$-cover
$(\widetilde{X}_\gamma,\widetilde{\omega}_\gamma)$ is not ergodic.
\end{theorem}

\begin{remark}
Forni in \cite{For-criterion} gives a concrete criterion for the positivity of Lyapunov exponents.
Eskin, Kontsevich and Zorich developed in \cite{EKZ1} an explicit formula
for all individual Lyapunov exponents for surfaces called square-tiled cyclic covers.
This gives tools for computing the number of positive exponents and then
for proving non-ergodicity for $\Z^d$-covers of such surfaces.
\end{remark}

Theorem \ref{thm:nonergo-easy} is a corollary of the more technical Theorem \ref{non-ergodicitycriterion} which gives a general criterion for non
ergodicity. This result, inspired by \cite{Fr-Ulc:nonerg}, relates the number of positive Lyapunov exponents of a compact translation surface to the non-ergodicity
of its $\Z^d$-covers. Theorem \ref{non-ergodicitycriterion} applies to generalized wind-tree models (see Theorem
\ref{thm:nonP} and Corollary \ref{cor:fin}).

Let $\Lambda$ be an arbitrary lattice in $\R^2$. Put a rectangular obstacle of size $(a,b)$ at each point in $\Lambda$ (see Figure
\ref{fig:windtree}). We stress that the sides of the obstacles are horizontal and vertical and not assumed to be parallel to any vector in $\Lambda$. We also
assume that the obstacles are pairwise disjoint. This defines an infinite billiard table $E(\Lambda,a,b)$ which is one of the most natural generalization of the wind-tree
model.
\begin{figure}[h]
\includegraphics[width=0.5\textwidth]{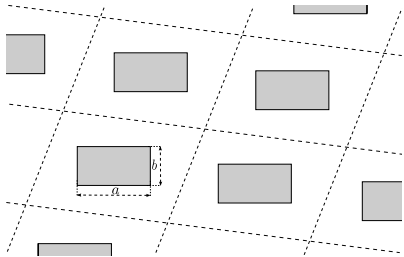}
\caption{Billiard table $E(\Lambda,a,b)$}\label{fig:windtree}
\end{figure}
As an application of Theorem \ref{non-ergodicitycriterion} (included in Corollary \ref{cor:fin} where more general models are considered)
we have that the directional billiard flow on $E(\Lambda,a,b)$ is not ergodic for almost every direction. This extends the main result of \cite{Fr-Ulc:nonerg} in which
 the case $\Lambda=\Z^2$ and almost all parameters $(a,b)$ is taken up.  We deal also with the problem of recurrence
on  $E(\Lambda,a,b)$ which was recently solved for $\Lambda=\Z^2$ in \cite{AH}.
As a partial solution we prove that following result:
\begin{theorem}\label{thm:recehr0}
For any real $\lambda$, let $\Lambda_\lambda = (1,\lambda)\Z + (0,1) \Z$.
For all $0<a,b<1$ and for every lattice $\Lambda_\lambda$, $\lambda\in\R$ the directional flow on $E(\Lambda_\lambda,a,b)$ is recurrent for a.e.\
direction.
\end{theorem}

Following the arguments in \cite{AH}, Theorem \ref{thm:recehr0} is obviously true for almost all triples $(\Lambda,a,b)$. We hope it is true for all parameters. To get an everywhere statement, a
classification of orbit closures in the small complexity would be an important step.

 \subsection{Organization of the paper}

Section~\ref{sect:covers} reviews background on translation surfaces and $\Z^d$-covers. Section~\ref{sect:cocycles} presents abstract results on skew
products and essential values. Section~\ref{Teich:sec} is a presentation of Teichm\"uller flow, Kontsevich-Zorich cocycle and Lyapunov exponents.
Section~\ref{sec:nonerggen} is devoted to the proof of Theorem~\ref{non-ergodicitycriterion} (non-ergodicity for some $\Z^d$-covers). Section~\ref{sec:recgen}
follows \cite{AH} and gives criteria for recurrence of linear flows in $\Z^d$-covers. Section~\ref{sect:square-tiled} applies the results of
Sections~\ref{sec:nonerggen}~and~\ref{sec:recgen} to square tiled surfaces. Sections~\ref{sec:genwt}~and~\ref{sect:rec-windtree} are devoted to applications to concrete examples.
Section~\ref{sec:genwt} gives non ergodicity results  for generalized wind-tree models based on  Section~\ref{sec:nonerggen}.
Section~\ref{sect:rec-windtree} applies the technics described in Section~\ref{sec:recgen} to prove recurrence in generalized wind-tree models.
Appendix~\ref{app} delivers the proof of two technical Theorems~\ref{cohthm}~and~\ref{cohcor}.

\section{$\Z^d$-covers of compact translation and half-translation surfaces} \label{sect:covers}
Let $X$ be a  connected oriented surface and let $q$ be a
quadratic meromorphic differential whose poles are of order at most one. The
pair $(X,q)$ is called a \emph{half-translation surface}.

If $q=\omega^2$ with $\omega$ an Abelian differential (holomorphic $1$-form) on $X$ then the pair $(X,\omega)$ is called a \emph{translation surface}.
Then we usually use the letter $M$ (instead of $X$) to denote the surface. Denote by $\Sigma\subset M$ the set of zeros of $\omega$. For every $\theta\in S^1 = \R/2\pi \Z $
denote by $X_\theta=X^{\omega}_\theta$ the directional vector
field in direction $\theta$ on $M\setminus\Sigma$.
Then the corresponding directional flow $(\varphi^{\theta}_t)_{t\in\R}=(\varphi^{\omega,\theta}_t)_{t\in\R}$ (also known as \emph{translation flow}) on $M\setminus\Sigma$
preserves the area measure $\mu_\omega$.

We will use the notation $(\varphi^{v}_t)_{t\in\R}$ and $X_v$ for
the \emph{vertical flow and vector field} (corresponding to
$\theta = \frac{\pi}{2}$) and $(\varphi^{h}_t)_{t\in\R}$ and $X_h$
for the \emph{horizontal flow and vector field} respectively
($\theta = 0$).

Let $(M,\omega)$ be a compact connected translation surface. A {\em $\Z^d$-cover} of $M$ is a surface $\widetilde{M}$ with a free totally discontinuous action
of the group $\Z^d$ such that the quotient manifold $\widetilde{M}/\Z^d$ is homeomorphic to $M$. Then the projection $p:\widetilde{M}\to M$  is called a
{\em covering map}. Denote by $\widetilde{\omega}$ the pullback of the form $\omega$ by the map $p$. Then $(\widetilde{M},\widetilde{\omega})$ is a translation
surface as well. All $\Z^d$-covers of $M$ up to isomorphism are in one-to-one correspondence with  $H_1(M,\Z)^d$. For any elements $\xi_1,\xi_2\in H_1(M,\Z)$ denote by $\langle \xi_1,\xi_2 \rangle $
the algebraic intersection number of $\xi_1$ with $\xi_2$. Then the  $\Z^d$-cover $\widetilde{M}_\gamma$
determined by $\gamma\in H_1(M,\Z)^d$ has the following properties:   if $\sigma$  is a close curve in $M$ and
\[n=(n_1,\ldots,n_d):=(\langle \gamma_1, [\sigma] \rangle,\ldots,\langle \gamma_d, [\sigma] \rangle ) \in
\mathbb{Z}^d\qquad([\sigma]\in H_1(M,\Z)),\] then $\sigma$ lifts
to a path $\widetilde{\sigma}: [t_0, t_1]\to \widetilde{M}_\gamma$
such that $\sigma(t_1) = n \cdot \sigma(t_0)$, where $\cdot$
denotes the action of $\Z^d$ on $\widetilde{M}_\gamma$.

\begin{remark}
We will always assume that $\Z^d$-cover $\widetilde{M}_\gamma$ is
not degenerated, i.e.\ elements $\gamma_1,\ldots,\gamma_d\in
H_1(M,\Z)$ establish an independent family in $H_1(M,\Q)$.
Otherwise, $\widetilde{M}_\gamma$ can be treated as a cover of
lower rank.
\end{remark}
\begin{remark}
Note that the surface $\widetilde{M}_\gamma$ is connected if and only if
the group homomorphism \[H_1(M,Z)\ni\xi\mapsto \langle\gamma,\xi\rangle:=\big(\langle\gamma_1,\xi\rangle,\ldots,\langle\gamma_d,\xi\rangle\big)\in \Z^d\]
is surjective.
\end{remark}
Let $(M,\omega)$ be a compact translation surface and let $(\widetilde{M},\widetilde{\omega})$ be its $\Z^d$-cover.
Let us consider a directional flows $(\widetilde{\varphi}^\theta_t)_{t\in\R}$ on $(\widetilde{M},\widetilde{\omega})$   such that
the flow $({\varphi}^\theta_t)_{t\in\R}$ on $(M,\omega)$ is minimal and ergodic. Let $I\subset M\setminus\Sigma$ be an interval transversal to the direction $\theta$ with no
self-intersections. Then the Poincar\'e return map $T:I\to I$ is a minimal ergodic interval exchange transformation (IET). Denote by $(I_\alpha)_{\alpha\in\mathcal{A}}$ the family of exchanged
intervals. For every $\alpha\in\mathcal{A}$ we will denote by $\xi_\alpha=\xi_\alpha(\omega,I)\in H_1(M,\Z)$ the homology class of any loop formed by the segment of orbit for
$(\varphi^\theta_t)_{t\in\R}$ starting at any $x\in\Int I_\alpha$ and ending at $Tx$ together with the segment of $I$ that joins $Tx$ and $x$, that we will
denote by $[Tx,x]$.

\begin{proposition}[see \cite{Fr-Ulc:nonerg} for $d=1$]\label{lem_flow_auto}
Let $I\subset M\setminus\Sigma$ be an interval transversal to the
direction $\theta$ with no self-intersections. Then  for every $\gamma\in H_1(M,\Z)^d$ the
directional flow $(\widetilde{\varphi}^\theta_t)_{t\in\R}$  on the
$\Z^d$-cover $(\widetilde{M}_\gamma,\widetilde{\omega}_\gamma)$
has a special representation over the skew product
$T_{\psi_\gamma}:I\times\Z^d\to I\times\Z^d$ of the form
$T_{\psi_\gamma}(x,n)=(Tx,n+\psi_\gamma(x))$, where
$\psi_\gamma:I\to\Z^d$ is a piecewise constant function given by
\begin{equation}\label{eq:psigamma}
\psi_\gamma(x)=\langle \gamma,\xi_\alpha\rangle=\big(\langle \gamma_1,\xi_\alpha\rangle,\ldots,\langle
\gamma_d,\xi_\alpha\rangle\big)  \quad\text{ if }\quad x\in
I_\alpha\quad\text{ for }\alpha\in\mathcal{A}.
\end{equation}
In particular, the ergodicity of the flow
$(\widetilde{\varphi}^\theta_t)_{t\in\R}$ on
$(\widetilde{M}_\gamma,\widetilde{\omega}_\gamma)$ is equivalent
to the ergodicity of the skew product $T_{\psi_\gamma}:I\times
\Z^d\to I\times\Z^d$.
\end{proposition}

\begin{notation}
We will use the notation $\psi_\gamma$ in a more general context
where $\gamma\in H_1(M,\R)^d$. Then $\psi_\gamma:I\to\R^d$ is also
a piecewise constant function given by \eqref{eq:psigamma}.
\end{notation}

\section{Cocycles for transformations and essential values.}\label{sect:cocycles}
Given  an ergodic automorphism $T$ of standard probability space
$\xbm$,  a locally compact abelian second countable group $G$ and
a measurable map $\psi:X \rightarrow G$, called a cocycle,
consider the skew-product extension \[T_\psi: (X\times
G,\mathcal{B}\times\mathcal{B}_G,\mu\times m_G) \rightarrow
(X\times G,\mathcal{B}\times\mathcal{B}_G,\mu\times m_G)\]
($\mathcal{B}_G$ is the Borel $\sigma$-algebra on $G$) given by \[
T_{\psi}(x,y)=(Tx,y+{\psi}(x)).\]
Clearly $T_\psi$ preserves the product of $\mu$ and  the Haar
measure $m_G$ on $G$. Moreover, for any $n\in\Z$ we have
\[ T^n_{\psi}(x,y)=(T^nx,y+{\psi}^{(n)}(x)), \]
where
\begin{equation}\label{cocycledef}
\psi^{(n)}(x)=\left\{
\begin{array}{ccl}
\psi(x)+\psi(Tx)+\ldots+\psi(T^{n-1}x) & \text{if} & n>0 \\
0 & \text{if} & n=0\\
-(\psi(T^nx)+\psi(T^{n+1}x)+\ldots+\psi(T^{-1}x)) & \text{if} &
n<0.
\end{array}
\right.
\end{equation}
Two cocycles $\psi_1,\psi_2: X\to G$ for $T$ are called {\em
cohomologous} if there exists a measurable function $g:X\to G$
(called the {\em transfer function}) such that
\[\psi_1=\psi_2+g-g\circ T.\] Then the corresponding skew products
$T_{\psi_1}$ and $T_{\psi_2}$ are measure-theoretically isomorphic
via the map $(x,y)\mapsto(x,y+g(x))$. A cocycle $\psi:X\to \R$ is
a {\em coboundary} if it is cohomologous to the zero cocycle.

\begin{definition} An element $g\in {G}$ is said to be an {\em
essential value} of $\psi:X\to G$, if for each open neighborhood
$V_g$ of $g$ in ${G}$ and each measurable set $B\subset X$ with
$\mu(B)>0$, there exists  $n\in\Z$ such that
\begin{eqnarray}
\mu(B\cap T^{-n}B\cap\{x\in X:\psi^{(n)}(x)\in V_g\})>0.
\label{val-ess}
\end{eqnarray}
\end{definition}

\begin{proposition}[see \cite{Sch}]
The set of essential values $E_G(\psi)$ is a closed subgroup of $G$ and the
skew product $T_{\psi}$ is ergodic if and only if $E_G(\psi)=G$.
\end{proposition}

The following result will be applied to prove that some skew products are not ergodic.
\begin{proposition}[see \cite{Sch}]\label{basicessentialvalues}
If $H$ is a
closed subgroup of $G$ and $\psi:X\to H$ then
$E_G(\psi)=E_H(\psi)\subset H$. If $\psi_1,\psi_2:X\to G$  are
cohomologous then ${E}_G(\psi_1)={E}_G(\psi_2)$. In particular, if $\psi:X\to G$ is a coboundary then $E_G(\psi)=\{0\}$.
\end{proposition}

\begin{lemma}\label{lem:skewpnonerg}
Suppose that $\psi:X\to\Z^d\subset \R^d$ is cohomologous, as
a cocycle taking values in $\R^d$, to a cocycle $\psi^*:X\to\R^d$
for which there exist $\alpha_1,\ldots,\alpha_d\in \R$ not all
rational such that the function
$\sum_{j=1}^d\alpha_j\psi^*_j:X\to\R$ takes only integer values.
Then the skew product $T_\psi$ is not ergodic.
\end{lemma}
\begin{proof}
Suppose that $m\in E_{\Z^d}(\psi)$. By
Proposition~\ref{basicessentialvalues}, we have
\[m\in E_{\Z^d}(\psi)=E_{\R^d}(\psi)=E_{\R^d}(\psi^*).\]
By assumption, the values of the cocycle $\psi^*$ belong to the
closed subgroup
\[G(\alpha):=\Big\{x\in\R^d:\sum_{j=1}^d\alpha_jx_j\in\Z\Big\},\]
so $E_{\R^d}(\psi^*)\subset G(\alpha)$. It follows that
$E_{\Z^d}(\psi)$ is a subgroup of $\Z^d\cap G(\alpha)$. Since at
least one numer $\alpha_1,\ldots,\alpha_d$ is irrational, the
group $\Z^d\cap G(\alpha)$ is a proper subgroup of $\Z^d$. Thus
$E_{\Z^d}(\psi)\neq \Z^d$ which yields non-ergodicity of the skew
product $T_{\psi}$.
\end{proof}
\subsection{Basic algebraic lemma}
The following result, together with Lemma~\ref{lem:skewpnonerg}, will help us (in Section~\ref{sec:nonerggen}) to relate
non-vanishing of some Lyapunov exponents of a Kontsevich-Zorich cocycle with non-ergodicity of a.e.\ directional flow on some $\Z^d$-cover of $(M,\omega)$.
\begin{lemma}\label{alglem}
Let $V$ be a real linear space of dimension $d_1+d_2$ and let
$a_1,\ldots, a_{d_1},$ $b_1,\ldots, b_{d_2}$ be its basis. Denote
by $V_\Z$ the lattice generated by the basis. Suppose that
$F\subset V$ is an $\R$-subspace of dimension $d_2$ such that
$F\cap V_\Z=\{0\}$. Then there exists real numbers
$\alpha_1,\ldots, \alpha_{d_1}$ not all rational and $c\in F$ such
that $\sum_{j=1}^{d_1}\alpha_ja_j+c\in V_\Z$.
\end{lemma}

\begin{proof}
Let us consider two linear maps $A:F\to\R^{d_1}$ and
$B:F\to\R^{d_2}$ determined by
\[c=\sum_{j=1}^{d_1}A_j(c)a_j+\sum_{j=1}^{d_2}B_j(c)b_j\quad\text{ for }\quad c\in F.\]
Next note that there exists a non-zero element $c\in F$ such that
$B(c)\in \Z^{d_2}$. Indeed, if $B$ is one-to-one then $B$ is a
bijection and we can choose any $c\in B^{-1}(\Z^{d_2}\setminus
\{0\})$. If $B$ is not one-to-one then we can take any $c\in \ker
B\setminus\{0\}$.

Fix such non-zero $c\in F$ with $B(c)\in \Z^{d_2}$. Then at least one real number $A_j(c)$,
$j=1,\ldots,d_1$ is irrational. Otherwise, multiplying vector $c$
by the least common multiple $m$ of the denominators of $A_j(c)$,
$j=1,\ldots,d_1$ we have
\[F\ni mc=\sum_{j=1}^{d_1}mA_j(c)a_j+\sum_{j=1}^{d_2}mB_j(c)b_j\in V_\Z,\]
which is impossible.
Finally, take $\alpha_j:=-A_j(c)$ for $j=1,\ldots,d_1$. Then
\[c+\sum_{j=1}^{d_1}\alpha_ja_j=\sum_{j=1}^{d_2}B_j(c)b_j\in V_\Z,\]
which is the desired conclusion.
\end{proof}

\section{The Teichm\"uller flow and the Kontsevich-Zorich cocycle}\label{Teich:sec}
Given  a connected compact oriented surface $M$,
denote by $\operatorname{Diff}^+(M)$ the
group of orientation-preserving homeomorphisms of $M$. Denote by $\operatorname{Diff}_0^+(M)$ the
subgroup of elements $\operatorname{Diff}^+(M)$ which are
isotopic to the identity. Let us denote by
$\Gamma(M):=\operatorname{Diff}^+(M)/\operatorname{Diff}_0^+(M)$
the {\em mapping-class} group. We will denote by $\mathcal{T}(M)$
(respectively $\mathcal{T}_1(M)$ ) the {\em Teichm\"uller space of
Abelian differentials } (respectively of unit area Abelian
differentials), that is the space of orbits of the natural action
of $\operatorname{Diff}_0^+(M)$ on the space of all
Abelian differentials on $M$ (respectively, the ones with total
area $\mu_\omega(M)=1$). We
will denote by $\mathcal{M}(M)$ ($\mathcal{M}_1(M)$) the {\em
moduli space of (unit area) Abelian differentials}, that is the
space of orbits of the natural action of
$\operatorname{Diff}^+(M)$ on the space of (unit area)
Abelian differentials on $M$. Thus
$\mathcal{M}(M)=\mathcal{T}(M)/\Gamma(M)$ and
$\mathcal{M}_1(M)=\mathcal{T}_1(M)/\Gamma(M)$.

The group $SL(2,\R)$ acts naturally on  $\mathcal{T}_1(M)$ and $\mathcal{M}_1(M)$ as follows.  Given a translation structure $\omega$, consider the charts
given by  local primitives of the holomorphic $1$-form. The new charts defined by postcomposition of this charts with an element of $SL(2,\R)$ yield a new
complex structure and a new differential which is Abelian with respect to this new complex structure, thus a new translation structure.
We denote by $g\cdot \omega$ the translation structure  on $M$
obtained acting by $g \in SL(2,\R)$ on a translation structure
$\omega$ on $M$.

The {\em Teichm\"uller flow} $(g_t)_{t\in\R}$ is the restriction
of this action to the diagonal subgroup
$(\operatorname{diag}(e^t,e^{-t}))_{t\in\R}$ of $SL(2,\R)$ on
$\mathcal{T}_1(M)$ and $\mathcal{M}_1(M)$. We will deal also with
the rotations $(r_{\theta})_{\theta\in S^1}$ that acts on
$\mathcal{T}_1(M)$ and $\mathcal{M}_1(M)$ by
$r_\theta\omega=e^{i\theta}\omega$.


\subsubsection*{Kontsevich-Zorich cocycle.}
 The {\em Kontsevich-Zorich (KZ) cocycle}
$(G^{KZ}_t)_{t\in\R}$ is the quotient of the trivial cocycle
\[g_t\times\operatorname{Id}:\mathcal{T}_1(M)\times H_1(M,\R)\to\mathcal{T}_1(M)\times H_1(M,\R)\]
by the action of the mapping-class group
$\Gamma(M):=\Gamma(M,\emptyset)$. The mapping class group acts on
the fiber $H_1(M,\R)$ by induced maps. The cocycle
$(G^{KZ}_t)_{t\in\R}$ acts on the cohomology vector bundle
\[\mathcal{H}_1(M,\R)=(\mathcal{T}_1(M)\times H_1(M,\R))/\Gamma(M)\]
 over the Teichm\"uller flow
$(g_t)_{t \in \R}$ on the moduli space $\mathcal{M}_1(M)$.

Clearly the fibers of the  bundle $\mathcal{H}_1(M,\R)$ can be
identified with $H_1(M,\R)$.  The space $H_1(M,\R)$ is endowed
with the symplectic form given by the algebraic intersection
number. This symplectic structure  is preserved by the action of
the mapping-class group and hence is invariant under the action of
$SL(2,\R)$.

The standard definition of KZ-cocycle is via the cohomological
bundle. The identification of the homological and cohomological
bundle and the corresponding KZ-cocycles is established by the
Poincar\'e duality $\mathcal{P}:H_1(M,\R)\to H^1(M,\R)$. This
correspondence allow us to define so called Hodge norm (see
\cite{For-dev} for cohomological bundle) on each fiber of the
bundle $\mathcal{H}_1(M,\R)$. The norm on the fiber $H_1(M,\R)$
over $\omega\in\mathcal{M}_1(M)$ will be denoted by
$\|\,\cdot\,\|_\omega$.

\subsubsection*{Lyapunov exponents and Oseledets splitting.}
Let  $\omega\in\mathcal{M}_1(M)$ and denote by
$\mathcal{M}=\overline{SL(2,\R)\omega}$ the closure of the $SL(2,\R)$-orbit of $\omega$ in $\mathcal{M}_1(M)$. The celebrated result of Eskin, Mirzakhani and Mohammadi, proved in
\cite{EMM} and \cite{EM}, says that $\mathcal{M}\subset
\mathcal{M}_1(M)$ is an affine $SL(2,\R)$-invariant submanifold.
Denote by $\nu_{\mathcal{M}}$ the corresponding affine
$SL(2,\R)$-invariant probability measure supported on
$\mathcal{M}$. Moreover, the measure $\nu_{\mathcal{M}}$ is ergodic under
the action of the Teichm\"uller flow. It follows that $\nu_{\mathcal{M}}$-almost every element of $\mathcal{M}$
is Birkhoff generic, i.e.\ pointwise ergodic theorem hold for the Teichm\"uller flow and every continuous integrable function on $\mathcal{M}$.
The following recent result is more refined and yields Birkhoff generic elements among $r_\theta\omega$  for $\theta\in S^1$.

\begin{theorem}[see \cite{Es-Ch}]\label{thm:esch}
For every $\phi\in C_c(\mathcal{M}_1(M))$ and almost all
$\theta\in S^1$ we have
\begin{equation}\label{eq:birk}
\lim_{T\to\infty}\frac{1}{T}\int_0^T\phi(g_tr_\theta\omega)\,dt=\int_{\mathcal{M}}\phi\,d\nu_{\mathcal{M}}.
\end{equation}
\end{theorem}

All directions $\theta\in S^1$ for which the assertion of the
theorem holds are called \emph{Birkhoff generic}.

Suppose that $V\subset H_1(M,\R)$ is a symplectic subspace (the
symplectic form restricted to $V$ is non-degenerated) of dimension
$2d$. Moreover, assume that $V$ is invariant for $SL(2,\R)$ action
on $\mathcal{M}$. Then $V$ defines a subbundle, denoted by
$\mathcal{V}$, of the bundle $\mathcal{H}_1(M,\R)$ over
$\mathcal{M}$ for which the fibers are identified with $V$.

Let us consider the KZ-cocycle $(G_t^{\mathcal{V}})_{t\in\R}$
restricted to $V$. By Oseledets' theorem, there exists Lyapunov
exponents of $(G_t^{\mathcal{V}})_{t\in\R}$ with respect to the
measure $\nu_{\mathcal{M}}$. As the action of the
Kontsevich-Zorich cocycle is symplectic, its Lyapunov exponents
with respect to the measure $\nu_{\mathcal{M}}$ are:
\[\lambda^{\mathcal{V}}_1\geq\lambda^{\mathcal{V}}_2\geq\ldots\geq\lambda^{\mathcal{V}}_d\geq-
\lambda^{\mathcal{V}}_d\geq\ldots\geq-\lambda^{\mathcal{V}}_2\geq-\lambda^{\mathcal{V}}_1\]

\begin{theorem}[see \cite{Es-Ch}]
Let
$\lambda^{\mathcal{V}}_1=\overline{\lambda}_1>\overline{\lambda}_2>
\ldots>\overline{\lambda}_{s-1}>\overline{\lambda}_s=-\lambda^{\mathcal{V}}_1$
be different Lyapunov exponents of the  Kontsevich-Zorich cocycle
$(G_t^{\mathcal{V}})_{t\in\R}$ with respect to the measure
$\nu_{\mathcal{M}}$. For almost all $\theta\in S^1$ there exists a
direct splitting
$V=\bigoplus_{i=1}^s\mathcal{U}_i(r_\theta\omega)$ such that for
every $\xi\in \mathcal{U}_i(r_\theta\omega)$ we have
\begin{equation}
\lim_{t\to\infty}\frac{1}{t}\log\|\xi\|_{g_tr_\theta\omega}=\overline{\lambda}_i.
\end{equation}
\end{theorem}
All directions $\theta\in S^1$ for which the assertion of the
theorem holds are called \emph{Oseledets generic}.
Then $V$ has a  direct splitting
\[ V=E_{r_\theta\omega}^+\oplus E_{r_\theta\omega}^0\oplus E_{r_\theta\omega}^-\]
into unstable, central and stable subspaces
\begin{align*}
E_{r_\theta\omega}^+&=\Big\{\xi\in
V:\lim_{t\to+\infty}\frac{1}{t}\log\|\xi\|_{g_{-t}r_\theta\omega}<0\Big\},
\label{stabledef}\\
E_{r_\theta\omega}^0&=\Big\{\xi\in
V:\lim_{t\to\infty}\frac{1}{t}\log\|\xi\|_{g_{t}r_\theta\omega}=0\Big\},\nonumber
\\
E_{r_\theta\omega}^-&=\Big\{\xi\in
V:\lim_{t\to+\infty}\frac{1}{t}\log\|\xi\|_{g_{t}r_\theta\omega}<0\Big\}.\nonumber
\end{align*}
The dimension of the stable and the unstable subspace is equal to the
number of positive Lyapunov exponents of
$(G^{\mathcal{V}}_t)_{t\in\R}$.

\begin{theorem}[see \cite{KMS}]\label{thm:KMS}
For every Abelian differential $\omega$ on a compact connected surface $M$ for almost all
directions $\theta\in S^1$ the directional flows
$(\varphi^{v}_t)_{t\in\R}$ and $(\varphi^{h}_t)_{t\in\R}$ on $(M,r_\theta\omega)$ are uniquely ergodic.
\end{theorem}
Every  $\theta\in S^1$ for which the assertion of the
theorem holds is called \emph{Masur generic}.

\smallskip

Suppose that the flow $({\varphi}^\theta_t)_{t\in\R}$ on
$(M,\omega)$ is uniquely ergodic. Denote by $M^+_\theta$ the set
of points $x\in M$ such that  the positive semi-orbit
$(\varphi^\theta_t(x))_{t\geq 0}$ on $(M,\omega)$ is well defined.
For every $x\in M^+_\theta$ and $t>0$ denote by
$\sigma^{\theta}_t(x)$ an element of $H_1(M,\Z)$ which is the
class of a loop formed by the orbit segment of the orbit of $x$ in
direction $\theta$ from $x$ to $\varphi^\theta_t(x)$ closing by
the shortest curve joining $\varphi^\theta_t(x)$ with $x$.

The following two results are closely related to Theorem~2 in
\cite{DHL} and Theorem~4.2 and Lemma~6.3 in \cite{Fr-Ulc:nonerg}.
For completeness of exposition we include their proofs in
Appendix~\ref{app}.
\begin{theorem}
\label{cohthm} Let $\omega\in\mathcal{M}_1(M)$. Suppose that
$\pi/2-\theta\in S^1$ is Birkhoff, Oseledets and Masur (BOM) generic. If $\gamma\in
E_{r_{\pi/2-\theta}\omega}^-$ then there exists $C>0$ such that $|\langle\sigma^{\theta}_t(x),
\gamma\rangle|\leq C$ for all $x\in M^+_\theta$ and $t>0$.  If additionally $\gamma\in
E_{r_{\pi/2-\theta}\omega}^-$ is non-zero and
there exists an orthogonal symplectic splitting $H_1(M,\Q)=K\oplus
K^{\perp}$ such that $V=\R\otimes_{Q}K$ and  $(\xi_i)_{i=1}^{2d}$
is a basis of $K$ then
$(\langle\gamma,\xi_i\rangle)_{i=1}^{2d}\notin \R \cdot \Q^{2d}$.
\end{theorem}


\begin{theorem}\label{cohcor}
Let $\omega\in\mathcal{M}_1(M)$. Suppose that
$\pi/2-\theta\in S^1$ is BOM generic. Then there exists $I\subset M\setminus\Sigma$  an
interval transversal to the direction $\theta$ with no
self-intersections such that the Poincar\'e return map $T:I\to I$ is a
minimal ergodic IET and
if $\gamma\in E_{r_{\pi/2-\theta}\omega}^-$ then the corresponding cocycle
$\psi_{\gamma,I}:I\to\R$ for $T$ is a coboundary.
\end{theorem}

\subsection{Lyapunov exponents for quadratic differentials and Eskin-Kontsevich-Zorich formula}
Let $(X,q)$ be a half-translation compact connected surface. Let $(M,\omega)$ be its canonical double cover.
Then $(M,\omega)$ is a compact translation surface for which there exists a holomorphic map  $\varrho:M\to X$ and
a holomorphic involution $I:X\to X$ such that
\[\varrho^*(q)=\omega^2,\quad \varrho\circ I=\varrho\ \text{ and }\ I^*(\omega)=-\omega.\]
%
The space $H_1(M,\R)$ has an orthogonal (symplectic) splitting into
\[H^+_1\!(M,\R)=\{\xi\in H_1(M,\R)\!:\!I_*\xi=\xi\}\text{ and }H^-_1\!(M,\R)=\{\xi\in H_1(M,\R)\!:\!I_*\xi=-\xi\}.\]
Moreover, $H^+_1(M,\R)$
is canonically isomorphic to $H_1(X,\R)$ via the map
$\varrho_*:H^+_1(M,\R)\to H_1(X,\R)$, we will identify
both spaces.

Let $\mathcal{M}=\overline{SL(2,\R)\omega}$ and let $\nu_{\mathcal{M}}$ be the corresponding affine
$SL(2,\R)$-invariant probability measure supported on $\mathcal{M}$.
The space $H^+_1(M,\R)$ defines a subbundle $\mathcal{H}^+_1(M,\R)$ over $\mathcal{M}$ and the Lyapunov
exponents of the Kontsevich-Zorich cocycle on this bundle are called Lyapunov exponents of $(X,q)$.

Let $\check{g}\geq 0$ be the genus of $X$ and let
$g_e:=g-\check{g}$. Then $2\check{g}=\dim H^+_1(M,\R)$ and
$2{g}_e=\dim H^-_1(M,\R)$. Both symplectic subspaces define
subbundles $\mathcal{H}^+_1(M,\R)$ and $\mathcal{H}^-_1(M,\R)$
respectively over $\mathcal{M}$. We will denote by
\[\lambda^{+}_1,\ldots,\lambda^+_{\check{g}},-\lambda^+_{\check{g}},\ldots,-\lambda^+_{1}\quad\text{ and }
\quad\lambda^{-}_1,\ldots,\lambda^-_{g_e},-\lambda^-_{g_e},\ldots,-\lambda^-_{1}\]
their Lyapunov exponents respectively.

Denote by $\mathcal{Q}(d_1,\ldots, d_n)$ the stratum of quadratic
differentials with singularities of angles $(2 + d_1)\pi, \ldots ,
(2 + d_n)\pi$ which are not squares of  Abelian differentials.
Recall that if $(X,q)\in \mathcal{Q}(d_1,\ldots, d_n)$ then
$d_1+\ldots+d_n=4\check{g}-4$.

\begin{theorem}[Eskin-Kontsevich-Zorich formula, see
\cite{EKZ2}]\label{thm:ekz}
If $(X,q)\in \mathcal{Q}(d_1,\ldots,
d_n)$ then
\begin{equation}\label{EKZformula}
\big(\lambda^{-}_1+\ldots+\lambda^-_{g_e}\big)-\big(\lambda^{+}_1+\ldots+\lambda^+_{\check{g}}\big)=\frac{1}{4}\sum_{j\text{
with }d_j\text{ odd}}\frac{1}{2+d_j}.
\end{equation}
\end{theorem}

\section{Non-ergodicity}\label{sec:nonerggen}
In this section we show how positivity of Lyapunov exponents affects the ergodic properties
of directional flows on $\Z^d$-covers of a translation surface $(M,\omega)$. Let $\mathcal{M}=\overline{SL(2,\R)\omega}$ and let $\nu_{\mathcal{M}}$ be the corresponding affine
$SL(2,\R)$-invariant probability measure supported on $\mathcal{M}$.
Assume that
\[ H_1 (M,\mathbb{Q}) = K \oplus K^\perp\]
is an orthogonal splitting (with respect to the symplectic
intersection form) such that $\dim_{\Q}{K}=2d\geq 2$. Suppose that
the subspace $V:=\R\otimes_{\Q}{K}\subset H_1(M,\R)$ defines an
$SL(2,\R)$-invariant subbundle $\mathcal{V}$ of the homological
bundle over $\mathcal{M}$. Suppose that  the number $d_+$ of positive
Lyapunov exponents for $(G^{\mathcal{V}}_t)_{t\in\R}$ is positive. Then the number of non-positive exponents
$d_-:=2d-d_+$ is less than $2d=\dim_{\Q}{K}$.

\begin{theorem}\label{non-ergodicitycriterion}
Let $\omega\in\mathcal{M}_1(M)$ and let $\pi/2-\theta\in S^1$
be BOM generic. Then, for  any non-degenerated
$\mathbb{Z}^{d_-}$-cover $(\widetilde{M}_\gamma,
\widetilde{\omega}_\gamma)$ of $(M, \omega)$ given by  $\gamma \in
(K\cap H_1(M,\Z))^{d_-}$, the directional flow
$(\widetilde{\varphi}^\theta_t)_{t\in \R}$ on
$(\widetilde{M}_\gamma, \widetilde{\omega}_\gamma)$ is not
ergodic. In particular, for a.e.\ $\theta\in S^1$
$(\widetilde{\varphi}^\theta_t)_{t\in \R}$ on
$(\widetilde{M}_\gamma, \widetilde{\omega}_\gamma)$ is not
ergodic.
\end{theorem}

\begin{proof}
Let $\gamma=(\gamma_{1},\ldots,\gamma_{d_-})$ with $\gamma_j\in
K\cap H_1(M,\Z)$ for $j=1,\ldots,d_-$. By assumption,
$\gamma_{1},\ldots,\gamma_{d_-}$ are independent in $K\subset
H_1(M,\Q)$. Therefore, we can find $\sigma_1,\ldots,\sigma_{d_+}$
in $K\cap H_1(M,\Z)$ such that
$\gamma_{1},\ldots,\gamma_{d_-},\sigma_1,\ldots,\sigma_{d_+}$
establish a basis in $K$ and hence in $V$. By assumption, the
stable subspace $E_{r_{\pi/2-\theta}\omega}^-\subset V$ has
dimension $d_+>0$. Moreover, $E_{r_{\pi/2-\theta}\omega}^-\cap
K=\{0\}$. Indeed, let $\xi\in
E_{r_{\pi/2-\theta}\omega}^-\cap K$. Since the collection
$\gamma_{1},\ldots,\gamma_{d_-},\sigma_1,\ldots,\sigma_{d_+}$ is a
basis of $K\subset H_1(M,\Q)$ and the algebraic intersection number of $\xi\in H_1(M,\Q)$
with every such base element is rational, in view of Theorem~\ref{cohthm},
$\xi\in E_{r_{\pi/2-\theta}\omega}^-$ must be trivial.

Since $d_-+d_+=2d=\dim_\R V$, by Lemma~\ref{alglem}, there exist
real numbers $\alpha_1,\ldots,\alpha_{d_-}$ not all rational
(assume, with no lost of generality, that $\alpha_1$ is
irrational) and $\xi\in E_{r_{\pi/2-\theta}\omega}^-$ such that
\[\sum_{j=1}^{d_-}\alpha_j\gamma_j+\alpha_1\xi\in H_1(M,\Z).\]

Next let us choose an interval $I\subset M\setminus \Sigma$ such that the corresponding IET $T:I\to I$ satisfies
the assertion of Theorem~\ref{cohcor}. In the remainder of we show that the skew product $T_{\psi_{\gamma}}:I\times\Z^{d_-} \to I\times\Z^{d_-}$
is not ergodic. In view of Proposition~\ref{lem_flow_auto}, this gives non-ergodicity  of the flow
$(\widetilde{\varphi}^\theta_t)_{t\in \R}$ on
$(\widetilde{M}_\gamma, \widetilde{\omega}_\gamma)$.

Denote by $(I_\alpha)_{\alpha\in\mathcal{A}}$ the
family of  intervals exchanged by $T$ and let
$(\xi_\alpha)_{\alpha\in\mathcal{A}}$ be the family in $H_1(M,\Z)$
defined just before Proposition~\ref{lem_flow_auto}.

Since $\xi\in E_{r_{\pi/2-\theta}\omega}^-$, by Theorem~\ref{cohthm}, the cocycle $\psi_\xi:I\to\R$ is a
coboundary, hence $(\psi_\xi,0,\ldots,0):I\to\R^{d_-}$ is also a
coboundary. Let $\psi^*:=\psi_\gamma+(\psi_\xi,0,\ldots,0)$ and
let us consider the  cocycle
$\sum_{j=1}^{d_-}\alpha_j\psi_j^*:I\to\R$. If $x\in I_\alpha$ then
\begin{eqnarray*}
\sum_{j=1}^{d_-}\alpha_j\psi_j^*(x)&=&\sum_{j=1}^{d_-}\alpha_j\psi_{\gamma_j}(x)+\alpha_1\psi_{\xi}(x)=
\sum_{j=1}^{d_-}\alpha_j\langle\gamma_j,\xi_\alpha\rangle+\alpha_1\langle\xi,\xi_\alpha\rangle\\
&=&\Big\langle\sum_{j=1}^{d_-}\alpha_j\gamma_j+\alpha_1\xi,\xi_\alpha\Big\rangle.
\end{eqnarray*}
Since both homology classes
$\sum_{j=1}^{d_-}\alpha_j\gamma_j+\alpha_1\xi$, $\xi_\alpha$ belong
to $H_1(M,\Z)$, it follows that the values of the cocycle
$\sum_{j=1}^{d_-}\alpha_j\psi_j^*$ are only integer. As the
cocycles $\psi_\gamma$, $\psi^*$ are cohomologous and  not all numbers $\alpha_1,\ldots,\alpha_{d_-}$ are rational, in view of
Lemma~\ref{lem:skewpnonerg}, the skew product $T_{\psi_\gamma}$ is
not ergodic.
\end{proof}

\section{Recurrence}\label{sec:recgen}
In this section we present a general approach, based mainly on \cite{AH}, that helps to prove the recurrence of directional
flows on $\Z^d$-covers of a compact translation surface $(M,\omega)$. Let us consider the holonomy operator $\hol_\omega:H_1(M,\R)\to \C$ given by $\hol_\omega(\xi)=\int_{\xi}\omega$.
Recall that if $\gamma=(\gamma_1,\ldots,\gamma_d)\in H_1(M,\Z)^d$  and at least one $\gamma_j$ is not in $\ker(\hol_{\omega})$ then for almost every $\theta\in S^1$ the flow
$(\widetilde{\varphi}^\theta_t)_{t\in\R}$ on $(\widetilde{M}_\gamma,\widetilde{\omega}_\gamma)$ is not recurrent, see \cite{Ho-We}. From now on we will always assume that $\hol_{\omega}(\gamma_j)=0$ for all $1\leq j\leq d$.

Assume that
$H_1 (M,\mathbb{Q}) = K \oplus K^\perp$
is an orthogonal splitting such that
the subspace $V:=\R\otimes_{\Q}{K}\subset H_1(M,\R)$ defines an
$SL(2,\R)$-invariant subbundle $\mathcal{V}$ of the homological
bundle over $\mathcal{M}=\overline{SL(2,\R)\omega}$. Suppose that $K\subset \ker(\hol_{\omega})$.

\begin{proposition}[see  Proposition~2 in \cite{AH}] \label{prop:rec}
Suppose that $(M,\omega_*)\in\mathcal{M}$ has a vertical cylinder $C$
whose core $\sigma(C)\in K^\perp\cap H_1(M,Z)$. If the positive $(g_t)_{t\in\R}$
orbit of $(M,\omega)$ accumulates on $(M,\omega_*)$ then for any $\gamma\in (K\cap H_1(M,\Z))^d$
the vertical flow on the $\Z^d$-cover $(\widetilde{M}_\gamma,\widetilde{\omega}_\gamma)$
is recurrent.
\end{proposition}

\begin{corollary}\label{cor:recgen}
Suppose that $(M,\omega)$ has a cylinder $C$
whose core $\sigma(C)\in K^\perp\cap H_1(M,Z)$. If $\pi/2-\theta\in S^1$ is Birkhoff generic then for every $\gamma\in (K\cap H_1(M,\Z))^d$
the directional flow $(\widetilde{\varphi}^\theta_t)_{t\in\R}$
on the $\Z^d$-cover $(\widetilde{M}_\gamma,\widetilde{\omega}_\gamma)$
is recurrent.
\end{corollary}

\begin{proof}
Denote by $\theta_0\in S^1$ the direction of the core of the cylinder $C$.
Then $C$ is a vertical cylinder for $(M,r_{\pi/2-\theta_0}\omega)\in \mathcal{M}$.
Since $\pi/2-\theta\in S^1$ is Birkhoff generic, applying \eqref{eq:birk} sequence a sequence
$(\phi_k)_{k\geq 1}$ in $C_c(\mathcal{M})$ such that $(\supp(\phi_k))_{k\geq 1}$
is  a decreasing nested sequence of non-empty compact subsets with the intersection
$\{r_{\pi/2-\theta_0}\omega\}$, there exists $t_n\to+\infty$ such that
$g_{t_n}(r_{\pi/2-\theta}\omega)\to r_{\pi/2-\theta_0}\omega$.
By Proposition~\ref{prop:rec}, for any $\gamma\in (K\cap H_1(M,\Z))^d$
the vertical flow on the $\Z^d$-cover $(\widetilde{M}_\gamma,\widetilde{r_{\pi/2-\theta}\omega}_\gamma)$
is recurrent. Since $\widetilde{r_{\pi/2-\theta}\omega}_\gamma=r_{\pi/2-\theta}\widetilde{\omega}_\gamma$,
it follows that the directional flow $(\widetilde{\varphi}^\theta_t)_{t\in\R}$
on $(\widetilde{M}_\gamma,\widetilde{\omega}_\gamma)$
is recurrent.
\end{proof}

\section{Square-tiled surfaces} \label{sect:square-tiled}
In this section we apply the results of Sections~\ref{sec:nonerggen} and \ref{sec:recgen} to establish some
ergodic properties of directional flows on $\Z^d$-cover of square-tiled surfaces.

A translation surface $(M,\omega)$ is {\em
square-tiled}   if there exists a ramified cover $p:M\to\R^2/\Z^2$
unramified outside $0\in\R^2/\Z^2$ such that $\omega=p^*(dz)$.
Any compact square
tiled surface $(M,\omega)$ can be realized by gluing finitely
many squares of equal sides in $\mathbb{R}^2$ by
identifying each left vertical side of a square with a right vertical
side of some square and each top horizontal side with a bottom
horizontal side via translations.

Let $(M,\omega)$ be such a compact square-tiled surface of genus $g>1$ and
set $\Sigma'=p^{-1}(\{0\})\subset M$.
For $i$-th square of $(M,\omega)$, let $\sigma_i,\zeta_i\in
H_1(M,\Sigma',\Z)$ be the relative homology class of the path in
the $i$-th square from the bottom left corner to the bottom right
corner (horizontal) and to the upper left corner (vertical), respectively. Let
$\sigma=\sum\sigma_i\in H_1(M,\Z)$ and $\zeta=\sum\zeta_i\in
H_1(M,\Z)$.
Denote by $H_1^{(0)}(M,\Q)$  the kernel of the homomorphism $p_*:H_1(M,\Q)\to H_1(\R^2/\Z^2,\Q)$ and let
$H_1^{st}(M,\Q)=\Q\sigma\oplus\Q\zeta$.
Then
\[H_1(M,\Q)=H^{(0)}_1(M,\Q)\oplus H_1^{st}(M,\Q)\]
is an orthogonal symplectic splitting which is $SL(2,\R)$-invariant over the orbit closure
od $(M,\omega)$ which is equal o $SL(2,\R)\cdot\omega$. Moreover,
\begin{equation*}\label{kerhol}
\dim_\Q H_1^{(0)}(M,\Q)=2g-2\quad\text{ and }\quad H_1^{(0)}(M,\Q)=\ker(\hol)\cap H_1(M,\Q).
\end{equation*}

\begin{remark}
Suppose that the square-tiled surface $M$ is decomposed into one cylinder, i.e.\ there is a cylinder $C\subset M$ with $\overline{C}=M$. Then
the core $\sigma(C)\in H_1^{st}(M,\Q)$.
\end{remark}

Now directly from Corollary~\ref{cor:recgen} and Theorem~\ref{non-ergodicitycriterion} we obtain Theorem \ref{thm:nonergo-easy} and

\begin{corollary}
Let $(M,\omega)$ be a compact square-tiled surface decomposed into one cylinder. Then for every
$\gamma\in (\ker(\hol)\cap H_1(M,\Z))^d$ ($d\geq 1$)
and a.e.\ direction $\theta\in S^1$ the directional flow on the
$\Z^d$-cover $(\widetilde{M}_\gamma,\widetilde{\omega}_\gamma)$ is recurrent.
\end{corollary}

\begin{remark}
Historically, this Corollary was proven in \cite{Hu-Le-Tr} for the wind-tree model. It was a motivation to get  a more general criterion like
Corollary~\ref{cor:recgen}.
\end{remark}



\section{General wind-tree model}\label{sec:genwt}
Let us consider a connected and simply connected centrally
symmetric polygon whose sides are only vertical or horizontal, and
suppose that $P$. A lattice $\Lambda\subset\R^2$ is called
$P$-admissible if the translated polygons $P+\rho$, $\rho\in
\Lambda$ are pairwise disjoint.

In this section we will deal with general periodic versions of
wind-tree model, this is with the billiard flow on the planar
table which is the complement (in $\R^2$) of the
$\Lambda$-periodic array of $P$ scatterers (whenever $\Lambda$ is
$P$-admissible), see Figure~\ref{fig:lattice}. The billiard table
we will denote by $W(\Lambda,P)$. The billiard flow on
$W(\Lambda,P)$, i.e.\ the geodesic flow on $W(\Lambda,P)$,  we
will denote by $(b_t)_{t\in\R}$.
\begin{figure}[h]
\includegraphics[width=0.8\textwidth]{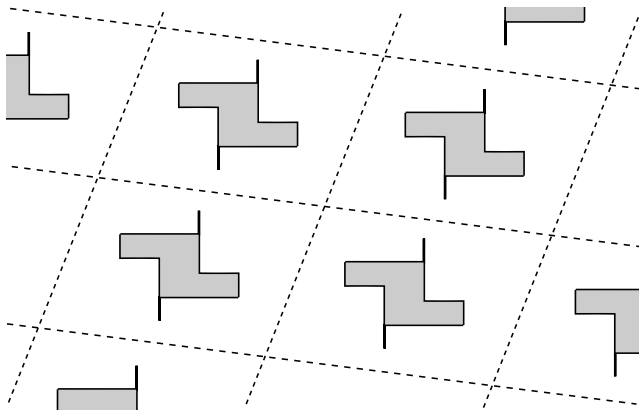}
\caption{Billiard table $W(\Lambda,P)$}\label{fig:lattice}
\end{figure}
Every trajectory of $(x,\theta)\in W(\Lambda,P)\times S^1$ for
$(b_t)_{t\in\R}$  travels in at most four directions, belonging to
the set $\Gamma \theta :=\{ \pm \theta, \pi\pm\theta\}$. The
restriction of $(b_t)_{t\in \R}$ to the invariant set
$W(\Lambda,P)\times\Gamma\theta$  will be denoted by
$(b_t^\theta)_{t\in\R}$ and called the directional billiard flow.
The flow $(b_t^\theta)_{t\in\R}$ preserves the product measure of
the Lebesgue measure on $W(\Lambda,P)$ and the counting measure on
$\Gamma\theta$ and   ergodic properties of $(b_t^\theta)_{t\in\R}$
refers to ergodic properties considered with respect to this
infinite measure.

Let us denote  by $\Gamma$ the $4$-elements group of isometries of
the plane generated by the horizontal and vertical reflections of
$\R^2$. Using the standard unfolding process, for every direction
$\theta\in S^1$ the flow $(b^{\theta}_t)_{t\in\R}$ on
$W(\Lambda,P)$ is isomorphic to the directional flow
$(\widetilde{\varphi}^{\theta}_t)_{t\in\R}$ on a non-compact
translation surface $\widetilde{M}(\Lambda,P)$. The surface
$\widetilde{M}(\Lambda,P)$ is obtained by gluing four copies of
$W(\Lambda,P)$, namely $\varrho \big(W(\Lambda,P)\big)$ for
$\varrho\in\Gamma$ so that each vertical and horizontal piece of
the boundary of any copy  $\varrho \big(W(\Lambda,P)\big)$ is
glued to is image by the vertical and horizontal reflection
respectively.
\begin{figure}[h]
\includegraphics[width=0.8\textwidth]{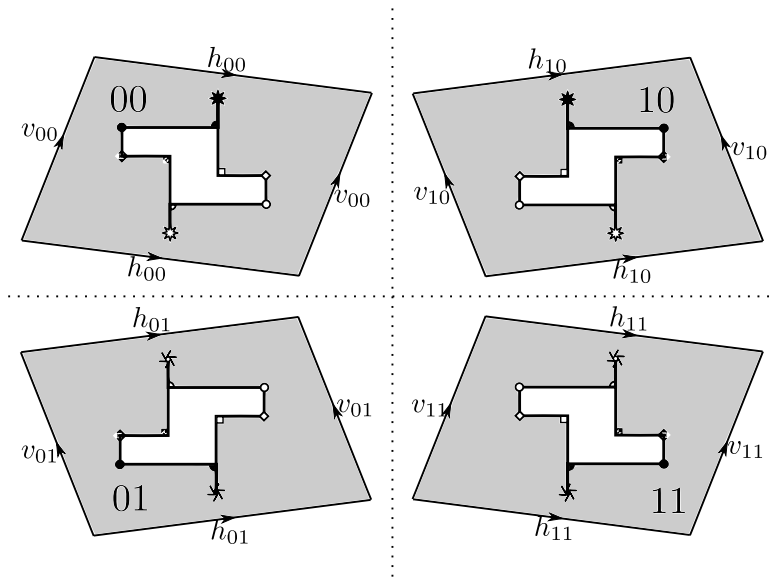}
\caption{Surface $(M,\omega)=M(\Lambda,P)$}\label{fig:surface2}
\end{figure}
The translation surface $\widetilde{M}(\Lambda,P)$ is a
$\mathbb{Z}^2$-cover of a compact translation surface
$(M,\omega)={M}(\Lambda,P)$, see in Figure \ref{fig:surface2}, and
the cover is given by the pair $\gamma=(\gamma_v,\gamma_h)\in
H_1({M}, \Z^2)$, where $\gamma_v=v_{00}-v_{10}+v_{01}-v_{11}$ and
$\gamma_h=h_{00}+h_{10}-h_{01}-h_{11}$ (referring to the labeling
of Figure \ref{fig:surface2}). The surface $M={M}(\Lambda,P)$ has
two natural involution $\tau:{M}\to{M}$, $\zeta_0:M\to M$. The map
$\tau$ acts by translations interchanging the part $00$ with $11$
and the part $10$ with $01$. The map $\zeta_0$ reflects each part
$00$, $01$, $10$ and $11$ in self through the center of symmetry
of the polygon. Then $\tau^*\omega=\omega$ and
$\zeta_0^*\omega=-\omega$. Let $\zeta_1:=\zeta_0\circ \tau$. Then
$\zeta_1:M\to M$ is an involution such that
$\zeta_1^*\omega=-\omega$. We will deal with two half-translation
surfaces $(X_0,q_0)=(M,\omega)/\langle\zeta_0\rangle$ and
$(X_1,q_1)=(M,\omega)/\langle\zeta_1\rangle$.

Let us consider all corners of the polygon $P$ and their full
exterior angles ($2\pi$ minus the angle). They can be equal
$\pi/2$, $3\pi/2$ and $2\pi$. Let $2m_j$ for $j=1,3,4$ be the
numbers of corners of $P$ with the exterior angle $j\cdot \pi /2$.
Then, computing the Euler characteristic of a piecewise flat annulus
(the complement in the complex plane of the connected polygon) we get:
\begin{equation} \label{Euler-Characteristic}
-2m_1+2m_3+4m_4=4.
\end{equation}
Each corner with exterior angle $\pi/2$ gives one singular point
on $(M,\omega)$ with angle $4\cdot\pi/2=2\pi$. Each corner with
exterior angle $3\pi/2$ gives one singular point on $(M,\omega)$
with angle $4\cdot 3\pi/2=6\pi$. Whereas, each corner with
exterior angle $2\pi$ gives two singular points on $(M,\omega)$
with angle $2\cdot 2\pi=4\pi$. Therefore,
$(M,\omega)\in\mathcal{H}(0^{2m_1},1^{4m_4},2^{2m_3})$ and the
genus of $M$ is $g=2m_3+2m_4+1$.

Let us look at the half-translation surface $(X_0,q_0)$. The
involution $\zeta_0$ does not fix any singular point (fake
singularities also). The only fixed points of $\zeta_0$ are
corners and centers of sides of parallelograms given by parts
$00$, $01$, $10$ and $11$ of $M$, there are $12$. Therefore,
$(X_0,q_0)\in\mathcal{Q}(0^{m_1},2^{2m_4},4^{m_3},(-1)^{12})$ and
the genus of $X_0$ is $g_0=m_3+m_4-2$.

The involution $\zeta_1$ fixes all singular point of degree $0$
and $2$ and does not fix singularities of degree $1$. Moreover,
there are no further fixed regular points.  Therefore,
$(X_1,q_1)\in\mathcal{Q}((-1)^{2m_1},2^{2m_4},1^{2m_3})$ and the
genus of $X_0$ is
\[g_1=\frac{1}{4}(2m_3+4m_4-2m_1+4)=2.\]
For $i=0,1$ the space $H_1(M,\Q)$ have a natural orthogonal
splitting into $\zeta_i$-invariant and $\zeta_i$-anti-invariant
parts, more precisely
\[H_1(M,\Q)=K_i\oplus K_i^\perp,\]
with
\[K_i:=\{\xi\in H_1(M,\Q):(\zeta_i)_*\xi=\xi\},\qquad K_i^\perp:=\{\xi\in H_1(M,\Q):(\zeta_i)_*\xi=-\xi\}.\]
Then
\[\dim K_0=2g_0=2(m_3+m_4-2),\quad \dim K_0^\perp=2g-2g_0=2(m_3+m_4+3), \]
\[\dim K_1=2g_1=2\cdot 2,\quad \dim K_1^\perp=2g-2g_1=2(2m_3+2m_4-1).\]
Moreover,
\[(\zeta_0)_*v_{ij}=-v_{ij}\ \text{ and }\ (\zeta_0)_*h_{ij}=-h_{ij}\ \text{ for all }i,j\in\{0,1\}\]
and
\[(\zeta_1)_*v_{00}=v_{11},\quad (\zeta_1)_*v_{01}=v_{10},\quad (\zeta_1)_*h_{00}=h_{11},\quad (\zeta_1)_*h_{01}=h_{10}.\]
It follows that
\[v_0:=v_{00}-v_{11},\quad v_1:=v_{10}-v_{01},\quad h_0:=h_{00}-h_{11},\quad h_1:=h_{10}-h_{01}\]
belong to both $K_1$ and $K_0^\perp$. Since
\[\langle h_i,v_j\rangle=2\delta_{i,j}\quad \text{and}\quad \langle h_i,h_j\rangle=\langle v_i,v_j\rangle=0 \]
for all $i,j\in\{0,1\}$, $h_0,h_1,v_0,v_1$ is a symplectic basis
of $K_1$. Hence $K_1\subset K_0^\perp$, so $K_0$ and $K_1$ are
symplectic orthogonal. It follows that $H_1(M,\Q)$ have an
orthogonal $SL(2,\R)$-invariant splitting
\[H_1(M,\Q)=K_1\oplus K_{0}\oplus K^\perp,\ \text{where}\ K^\perp:=K_0^{\perp}\cap K_1^\perp\ \text{and}\ \dim K^\perp=2(m_3+m_4+1).\]
All three subspaces $K_0,K_1,K^\perp$ define $SL(2,\R)$-invariant
symplectic bundles $\mathcal{V}_0$, $ \mathcal{V}_1$,
$\mathcal{V}^\perp$ over $\mathcal{M}=\overline{SL(2,\R)\omega}$.
Let us denote by
\[\lambda_1,\lambda_{2},-\lambda_{2},-\lambda_{1};\quad
\lambda^{1}_1,\ldots,\lambda^1_{m_3+m_4-2},-\lambda^1_{m_3+m_4-2},\ldots,-\lambda^1_{1};\]\[
\lambda^{\perp}_1,\ldots,\lambda^\perp_{m_3+m_4+1},-\lambda^\perp_{m_3+m_4+1},\ldots,-\lambda^\perp_{1}\]
their Lyapunov exponents respectively.

Now Theorem~\ref{thm:ekz} applied twice to the surface $(X_0,q_0)$
and to $(X_1,q_1)$ yields
\[\sum_{j=1}^{m_3+m_4+1}\lambda^\perp_j +\lambda_1+\lambda_2-\sum_{j=1}^{m_3+m_4-2}\lambda^1_j=\frac{1}{4}(12\cdot 1)=3\]
and
\[\sum_{j=1}^{m_3+m_4+1}\lambda^\perp_j +\sum_{j=1}^{m_3+m_4-2}\lambda^1_j-\lambda_1-\lambda_2=\frac{1}{4}\Big(2m_1\cdot 1+2m_3\cdot\frac{1}{3}\Big)=
\frac{2}{3}m_3+m_4-1.\] It follows that
\begin{equation}\label{eq:sumlambda}
\lambda_1+\lambda_2=
2-\frac{1}{3}m_3-\frac{1}{2}m_4+\sum_{j=1}^{m_3+m_4-2}\lambda^1_j.
\end{equation}

\begin{theorem}\label{thm:nonP}
Suppose that the polygon $P$ and the lattice $\Lambda$ are such that $\lambda_1>0$ and $\lambda_2>0$.
Then for a.e.\ direction $\theta\in S^1$ the  directional billiard flow
$(b^\theta_t)_{t\in\R}$ on $W(\Lambda,P)$ is not ergodic.
\end{theorem}

\begin{proof}
We need to show that for a.e.\ $\theta\in S^1$ the flow $(\widetilde{\varphi_t^\theta})_{t\in\R}$ on $\widetilde{M(\Lambda,P)}_\gamma$ is non-ergodic,
where $\gamma:=(\gamma_v,\gamma_h)\in H_1(M,\Z)^2$. To prove this, we apply Theorem~\ref{non-ergodicitycriterion} to the surface $(M,\omega)=M(\Lambda,P)$
and the symplectic $SL(2,\R)$-invariant splitting $H_1(M,\Q)=K_1\oplus K_1^\perp$. Recall that $\dim_\Q K_1=2d=4$, the Lyapunov exponents of the bundle $\mathcal{V}_1$ generated by $K_1$ are $\lambda_1,\lambda_2,-\lambda_2,-\lambda_1$ and
\[\gamma_v=v_0-v_1\in K_1\quad\text{ and }\quad \gamma_h=h_0+h_1\in K_1.\]
Since both exponents $\lambda_1,\lambda_2$ are positive, we have $d_+=2$, so $d_-=2d-d_+=2$. As $\gamma_v,\gamma_h\in K_1\cap H_1(M,\Z)$ and they are independent ($\langle\gamma_v,v_0\rangle=0\neq 2=\langle\gamma_h,v_0\rangle$), by  Theorem~\ref{non-ergodicitycriterion}, for a.e.\ $\theta\in S^1$ the flow $(\widetilde{\varphi_t^\theta})_{t\in\R}$ on $\widetilde{M(\Lambda,P)}_\gamma$ is non-ergodic.
\end{proof}

\begin{corollary}\label{cor:fin}
Suppose that the polygon $P$ is such that $\frac{1}{3}m_3+\frac{1}{2}m_4\leq 1$. If $\Lambda$ is $P$-admissible then
for a.e.\ direction $\theta\in S^1$ the  directional billiard flow
$(b^\theta_t)_{t\in\R}$ on $W(\Lambda,P)$ is not ergodic.
\end{corollary}
\begin{proof}
By \eqref{eq:sumlambda}, we have
\[\lambda_1+\lambda_2\geq 2-\frac{1}{3}m_3-\frac{1}{2}m_4\geq 1.\]
Since both $\lambda_1,\lambda_2$ are strictly less than $1$, it follows that $\lambda_1,\lambda_2$ are positive and Theorem~\ref{thm:nonP} can be applied.
\end{proof}

Let us collect all triples $(m_1,m_3,m_4)$ and central symmetric polygons $P$ for which
\[m_3+2m_4-m_1=2 \text{ (see \eqref{Euler-Characteristic})} \quad\text{and}\quad\frac{1}{3}m_3+\frac{1}{2}m_4\leq 1\quad\text{ are valid.}\]
First note that
\[\frac{m_1+2}{4}=\frac{m_3+2m_4}{4}=\frac{1}{4}m_3+\frac{1}{2}m_4\leq\frac{1}{3}m_3+\frac{1}{2}m_4\leq 1,\]
so $0\leq m_1\leq 2$. Therefore, there are only five possibilities: $(0,0,1)$, $(0,2,0)$, $(1,1,1)$, $(1,3,0)$, $(2,0,2)$.

The case $(0,0,1)$ is not interesting. Then $P$ is degenerated to an interval the dynamics of the wind-tree model is clear: a typical orbit escapes to infinity.

If $(m_1,m_3,m_4)=(0,2,0)$ then $P$ is a rectangle and pass to so called periodic Ehrenfest wind-tree model. If $P=[0,a]\times[0,b]$ then the corresponding billiard table we will denote by $E(\Lambda,a,b)$.
The recurrence, ergodicity and diffusion times of standard Ehrenfest wind-tree model, with  $\Lambda=\Z^2$, were studied recently, see \cite{AH, Co-Gu, Del,
DHL, Fr-Ulc:nonerg, Ho1, Hu-Le-Tr}. In particular,  it was recently
shown that for every pair of parameters $(a,b)$ and almost every
direction $\theta$ the billiard flow on $E(\Z^2,a,b)$ is recurrent
and non-ergodic and its rate of diffusion is $t^{2/3}$. By Corollary~\ref{cor:fin}, we have non-ergodicity for any lattice $\Lambda\subset
\R^2$ and every choice of parameters $(a,b)$.

If $m_1=1$ then there are four possible shapes of the polygon $P$, see Figure~\ref{fig:m11}
\begin{figure}[h]
\includegraphics[width=0.8\textwidth]{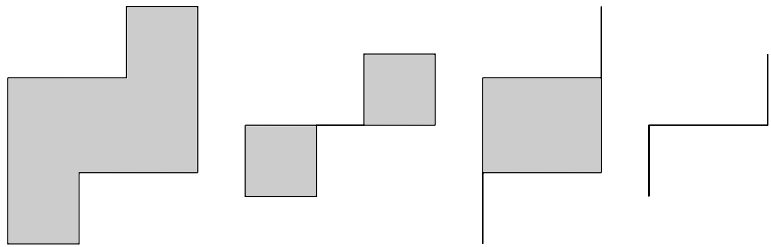}
\caption{Polygon $P$ for $(1,3,0)$ and $(1,1,1)$.}\label{fig:m11}
\end{figure}

If $m_1=2$ then there are also four possible shapes of the polygon $P$, see Figure~\ref{fig:m12}
\begin{figure}[h]
\includegraphics[width=0.8\textwidth]{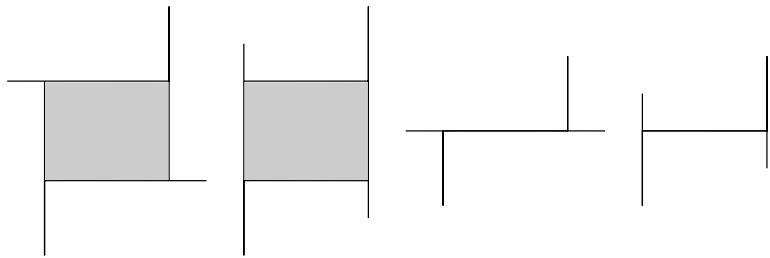}
\caption{Polygon $P$ for $(2,0,2)$.}\label{fig:m12}
\end{figure}

\section{Recurrence in general wind-tree model} \label{sect:rec-windtree}
In this section we study the problem of recurrence for a.e.\ directional billiard flow in general wind-tree model.
Recently, Avila and the second author proved recurrence for a.e.\ directional billiard flow on $E(\Z^2,a,b)$ for all $0<a,b<1$.
In Section~\ref{subsec:recehr}, we extend this result to lattices of the form  $\Lambda_\lambda:=(1,\lambda)\Z+(0,1)\Z$.

Section~\ref{subsec:concave} is devoted to the recurrence problem for polygons $P$ that have corners with full exterior angle $\pi/2$.

In both cases we will apply Corollary~\ref{cor:recgen}. We will look for specific (symmetric) cylinders in the surface  $M(\Lambda,P)$.
Such cylinders will have different origin in the above mentioned two models. In a sense, the existence of concave corners in a model
helps to the required cylinder.

\subsection{Ehrenfest wind-tree model}\label{subsec:recehr}
In this section we will prove recurrence in wind-tree model
$E(\Lambda,a,b)$ (for a.e.\ direction $\theta$)  for some families
of parameters $\lambda, a, b$. We deal with the lattices of the
form $\Lambda_\lambda:=(1,\lambda)\Z+(0,1)\Z$.

\begin{theorem}\label{thm:recehr}
For all $0<a,b<1$ and for every lattice $\Lambda_\lambda$, $\lambda\in\R$ the directional flow on $E(\Lambda_\lambda,a,b)$ is recurrent for a.e.\
direction.
\end{theorem}

\begin{proof}
Let $M=M(\Lambda_\lambda,a,b):=M(\Lambda_\lambda,[0,a]\times[0,b])$. Let us consider the splitting
$H_1(M,\Q)=K\otimes K^\perp$ with
\[K:=\{\xi\in H_1(M,\Q):\tau_*\xi=-\xi\},\qquad K^\perp:=\{\xi\in H_1(M,\Q):\tau_*\xi=\xi\},\]
where $\tau:M\to M$ is an involution  defined in Section~\ref{sec:genwt}. This splitting is $SL(2,\R)$-invariant and $\gamma_h,\gamma_v\in K$.
Therefore, in view of Corollary~\ref{cor:recgen}, we need to find  a regular $\bar{x}\in M$ and a $S$ segment of a translation orbit
joining $\bar{x}$ and $\tau(\bar{x})$. Indeed,  $S$ closed up by $\tau(S)$ gives a periodic regular orbit which is $\tau$-invariant. This periodic orbit is
a boundary of a cylinder $C$ for which $\tau_*\sigma(C)=\sigma(C)$, this is $\sigma(C)\in K^\perp$, and  Corollary~\ref{cor:recgen} applies.

We will find the segment $S$ by looking for segments of billiard orbits on  $E(\Lambda_\lambda,a,b)$ such as in right parts on
Figures~\ref{fig:full0}~and~\ref{fig:full1}.
\begin{figure}[h]
\includegraphics[width=0.8\textwidth]{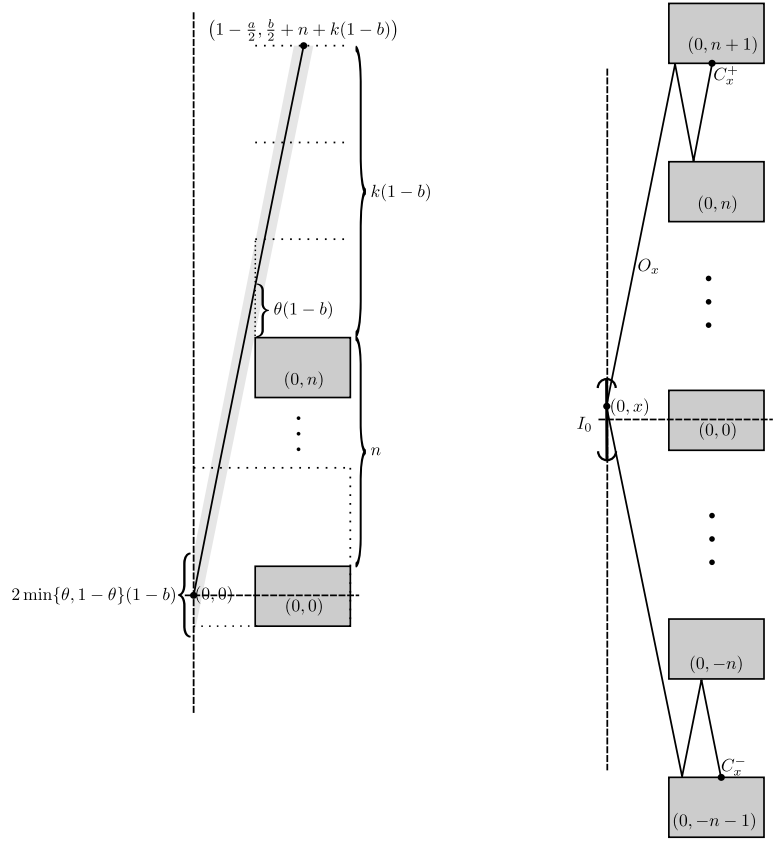}
\caption{The construction of segment $S$; step $0$.}\label{fig:full0}
\end{figure}
Fix $a,b\in(0,1)$. For convenience, assume that the lattice of rectangular scatterers consists of rectangles
\[[1-a,1]\times[-b/2,b/2]+(m_1,m_2+\lambda m_1)\quad\text{for}\quad (m_1,m_2)\in\Z^2.\]
Each such rectangle is labeled by $(m_1,m_2)\in \Z^2$, see on Figures~\ref{fig:full0}~and~\ref{fig:full1} where $0$-th column of the rectangles is drawn.

Fix $n_0\in\N$ such that
\begin{equation}\label{eq:zaln0}
n_0\geq\frac{2}{a}(1-b).
\end{equation}

\textbf{Step $0$.} Let us choose $n\geq n_0$, $k\geq 1$ and $0\leq\theta<1$ such that
\begin{equation}\label{def:nkt}
\gamma_1 k-\gamma_2 (n+b/2)=\theta\ \text{ with }\ \gamma_1:=\frac{2(1-a)}{2-a}\text{ and }\gamma_2:=\frac{a}{(2-a)(1-b)}.
\end{equation}
Since $0<\gamma_1<1$, the required triple $(n,k,\theta)$ does exist. Let $\gamma:=\min\{\gamma_1,\gamma_2\}$. Then a more careful choice yields a triple
$(n,k,\theta)$ with
\begin{equation}\label{ineq:theta}
\Big|\theta-\frac{1}{2}\Big|\leq \frac{\gamma}{2}.
\end{equation}
Moreover, since $\gamma_1+(1-b)\gamma_2=1$, we have
\begin{equation}\label{ineq:gamma}
\gamma\leq \frac{(1-b)\gamma_1+(1-b)\gamma_2}{2(1-b)}<\frac{\gamma_1+(1-b)\gamma_2}{2(1-b)}=\frac{1}{2(1-b)}.
\end{equation}

Let us consider the billiard orbit $O_0$ starting from the point $(0,0)$ with the slope \[s:=\frac{n+b/2+k(1-b)}{1-a/2},\] see Figure~\ref{fig:full0}.
Since, by \eqref{def:nkt},
\[\frac{n+b/2+k(1-b)}{1-a/2}(1-a)=n+b/2+\theta(1-b)\ \text{ and }\ \theta\in(0,1),\]
the orbit passes between the left sides of the $(0,n)$-th and the $(0,n+1)$-st rectangle and next it reflects alternately from the lower side of the
$(0,n+1)$-st rectangle and the upper side of the $(0,n)$-th rectangle. After $k$-th such reflection the orbit reaches the center, denoted by $(c_1,c_2)$,
of the lower side of the $(0,n+1)$-st rectangle (if $k$ is odd) or the upper side of the $(0,n)$-th rectangle (if $k$ is even), see Figure~\ref{fig:full0}.

Now let us change the starting point $(0,0)$ into $(0,x)$ with $x\in (-\theta(1-b),(1-\theta)(1-b))$. Then the billiard orbit $O_x$ flows parallel to $O_0$
until the $k$-th reflection and the distance in the horizontal direction between $O_x$ and  $O_0$ is
\[\frac{|x|}{s}\leq \frac{1-b}{s}=\frac{(1-a/2)(1-b)}{n+b/2+k(1-b)}\leq\frac{1}{2}\frac{(2-a)(1-b)}{n_0+b/2}.\]
Moreover, by \eqref{eq:zaln0},
\[n_0+b/2> \frac{2}{a}(1-b)+\frac{b-1}{2}=\frac{4-a}{2a}(1-b)\geq \frac{2-a}{a}(1-b).\]
Therefore, the horizontal distance $\frac{|x|}{s}\leq \frac{a}{2}$. It follows that after the $k$-th reflection $O_x$ reaches the point
\[C_x^+:=(c_1-x/s,c_2).\]

By symmetry, if \[x\in I_0:=\big(-\min\{\theta,1-\theta\}(1-b),\min\{\theta,1-\theta\}(1-b)\big)\] and  $(0,x)$ belongs to the right side of a rectangular
scatterer (from the $-1$-st column) then the backward orbit of $(0,x)$ after the $k$-th reflection  reaches the point
\[C_x^-:=(c_1+x/s,-c_2).\]
Note that the required point $x$ exists if the sets $\Z-\lambda+[-b/2,b/2]$ and $I_0$ have non-trivial intersection.

If $\bar{x}$ is the projection  of $C_x^-$ on $M(\Lambda_\lambda,a,b)$ then $\tau(\bar{x})$ is the projection of $C_x^+$. Therefore, the projection of the
segment of $O_x$ joining $C_x^-$ and $C_x^+$ yields the required linear segment $S$ joining $\bar{x}$ and $\tau(\bar{x})$ in $M(\Lambda_\lambda,a,b)$.
Moreover, in view of \eqref{ineq:theta} and \eqref{ineq:gamma}, we have
\begin{eqnarray*}|I_0|&=&2\min\{\theta,1-\theta\}(1-b)=2\Big(\frac{1}{2}-\big|\frac{1}{2}-\theta\big|\Big)(1-b)>(1-\gamma)(1-b)\\
&\geq&\Big(1-\frac{1}{2(1-b)}\Big)(1-b)=\frac{1}{2}-b.
\end{eqnarray*}
\begin{figure}[h]
\includegraphics[width=0.8\textwidth]{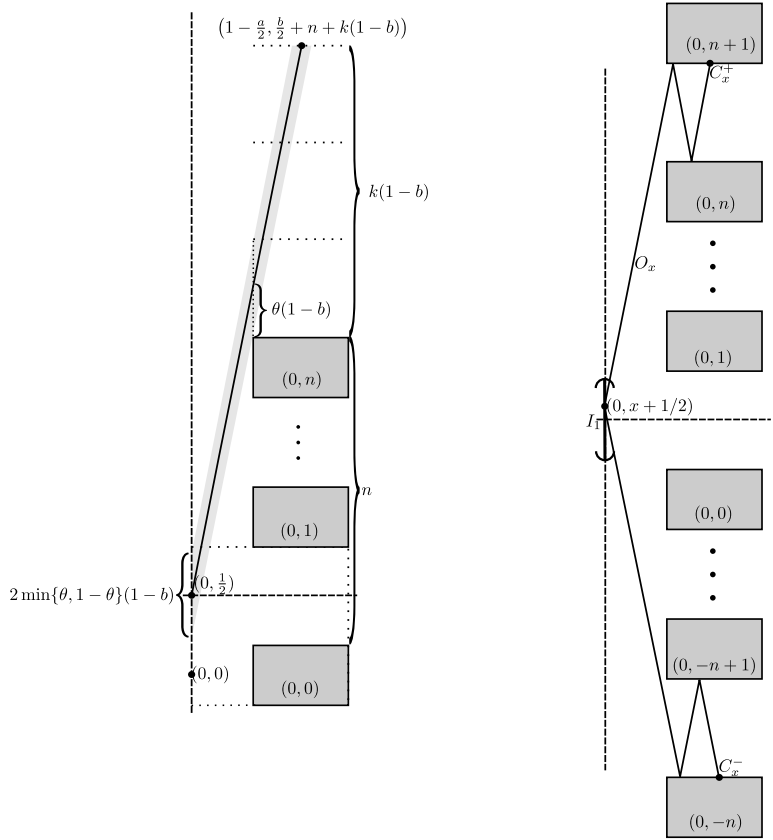}
\caption{The construction of segment $S$; step $1$.}\label{fig:full1}
\end{figure}

\textbf{Step $1$.} Similar constructions may be made changing the base point $(0,0)$ into $(0,1/2)$, see Figure~\ref{fig:full1}. Then we choose the triple
$(n,k,\theta)$ such that
\[\gamma_1k-\gamma_2(n+b/2-1/2)=\theta\ \text{ and }\ \big|\theta-\frac{1}{2}\big|\leq \frac{\gamma}{2}.\]
Moreover, the slope of the motion now is equal to
\[s:=\frac{n+b/2-1/2+k(1-b)}{1-a/2}.\]
Similar arguments to those from Step $0$ show that if
\[x+\frac{1}{2}\in I_1:=\frac{1}{2}+\big(-\min\{\theta,1-\theta\}(1-b),\min\{\theta,1-\theta\}(1-b)\big)\] and  $(0,x+1/2)$ belongs
to the right side of a rectangular scatterer then the forward and the backward orbit of $(0,x+1/2)$ (with the slope $s$)  after the $k$-th reflection
reaches the points $C_x^+:=(c_1-x/s,c_2)$ and $C_x^-:=(c_1+x/s,-c_2+1)$ respectively, where $(c_1,c_2)$ is the center of the lower side of the $(0,n+1)$-st
rectangle  or the upper side of the $(0,n)$-th rectangle, see Figure~\ref{fig:full1}. Note that the required point $x$ exists if the sets
$\Z-\lambda+[-b/2,b/2]$ and $I_1$ have non-trivial intersection. Moreover,
\[|I_1|=2\min\{\theta,1-\theta\}(1-b)>(1-\gamma)(1-b)\geq\frac{1}{2}-b.\]
Then the projection (on $M(\Lambda_\lambda,a,b)$) of the segment of the billiard orbit joining $C_x^-$ with $C_x^+$ is the required segment $S$.

\textbf{Final step.} As we have already proved, the required segment $S$ exists whenever the interval $[-/b/2,b/2]-\lambda$ intersects the set $I:=(I_0\cup
I_1)+\Z$. The set $I$ is the union of open intervals centered at points $m$, $m+1/2$ for all $m\in\Z$ and the length of each interval is greater than
$1/2-b$. It follows that the distance between consecutive intervals is less than $b$. Therefore, $[-/b/2,b/2]-\lambda$ always intersects the set $I$ which
completes the proof.
\end{proof}

\subsection{Ehrenfest models with suckers}\label{subsec:concave}
In this section we deal with billiard table $W(\Lambda,P)$ when $P$ has at least one concave corner (the first interesting example was discovered by Panov in \cite{Pa}).
The existence of such concave corners sometimes helps to find required cylinder which are need to prove recurrence, cf.\  Corollary~\ref{cor:recgen}.
More precisely, in this section we will prove recurrence (for a.e.\ direction) for three model formed form Ehrenfest model by attaching to the rectangular scatterers
some linear suckers. We will deal polygons presented on Figure~\ref{fig:rectwithsuck}.
\begin{figure}[h]
\includegraphics[width=1\textwidth]{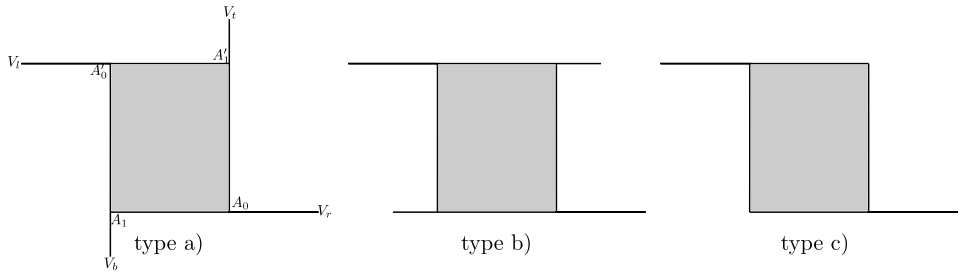}
\caption{Rectangles with suckers; types a), b) and c).}\label{fig:rectwithsuck}
\end{figure}

We start from a simple general observation that indicates how the existence of concave corners can help to prove recurrence.
Suppose that  $P$ is a centrally symmetric polygon. For every $Q\in P$ we denote by $Q'\in P$
the image of $Q$ by the central symmetry of $P$. In this section we will always assume that $0$ is the center of symmetry of $P$.

\begin{lemma}\label{lem:cyl_rec}
Suppose that $C$ is a concave corner of $P$ and there is a billiard segment $S$ in $W(\Lambda,P)$ joining $C$ and $C'+\rho$ for some $\rho\in \Lambda$
and such that $S$ does not meet any other corner in $W(\Lambda,P)$. Then for a.e.\
direction $\theta\in S^1$ the directional billiard flow
$(b^\theta_t)_{t\in\R}$ on $W(\Lambda,P)$ is recurrent.
\end{lemma}

\begin{proof}
Choose $\vep>0$ small enough such that the $\vep$-neighborhood of $S$ still does not meet any other corner in $W(\Lambda,P)$.
Let us take any element $C_\vep$ of a side coming from the corner $C$ and lying in the neighborhood. Then the billiard orbit
starting from $C_\vep$ goes parallel to $S$ and finally after two reflections reaches the point $C'_\vep+\rho$, see Figure~\ref{fig:concave}.
\begin{figure}[h]
\includegraphics[width=0.7\textwidth]{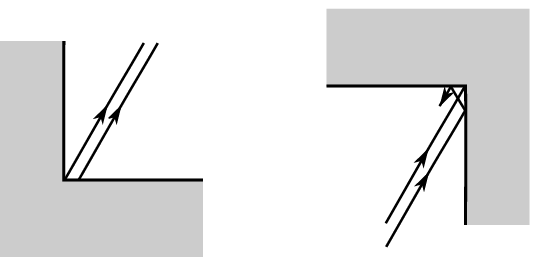}
\caption{Segments $S$ and $S_\vep$.}\label{fig:concave}
\end{figure}
Let us denote by $S_\vep$ the billiard orbit segment joining $C_\vep$ and $C'_\vep+\rho$. Its projection on $M(\Lambda,P)$ is a linear regular segment joining
a regular point on $M(\Lambda,P)$ and its image via $\tau$. The rest of the proof runs as at the beginning of the proof of Theorem~\ref{thm:recehr}.
\end{proof}

First assume that $P$ is a rectangle with suckers of type a), see Figure~\ref{fig:rectwithsuck}.
Let us denote by $V_r,V_l,V_t,V_b$ the convex and by $A_0,A'_0,A_1,A'_1$ the concave corners of $P$ as in Figure~\ref{fig:rectwithsuck}.
For every pair $(A_i,A'_i+\rho)$ ($\rho\in\Lambda$ and $i=0,1$) with
\[\Re(A_i'+\rho-A_i)>0\ \text{ and }\ (-1)^i\Im(A_i'+\rho-A_i)>0,\] denote by $R(A_i,A'_i+\rho)$
the open rectangle (with vertical and horizontal sides) with opposite corners $A_i$, $A'_i+\rho$. We call a pair $(A_i,A'_i+\rho)$ \emph{free} if $R(A_i,A'_i+\rho)$ does not intersect any polygon $P+\rho'$ for $\rho'\in\Lambda$, see Figure~\ref{fig:free_rect}.
\begin{figure}[h]
\includegraphics[width=0.5\textwidth]{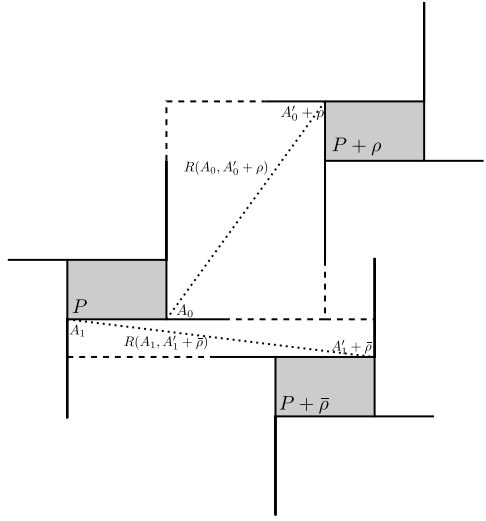}
\caption{Rectangles $R(A_0,A'_0+\rho)$ and $R(A_1,A'_1+\bar{\rho})$.}\label{fig:free_rect}
\end{figure}
\begin{remark}\label{rem:rec}
If a billiard table $W(\Lambda,P)$ has a free pair $(A_i,A'_i+\rho)$ then the corners $A_i$, $A'_i+\rho$ are joined by a straight billiard segment and, by Lemma~\ref{lem:cyl_rec}, a.e.\ direction on $W(\Lambda,P)$ is recurrent.
\end{remark}

\begin{lemma}\label{lem:typeA}
If $P$ is a rectangle with suckers of type a) then for every $P$-admissible lattice $\Lambda$ the billiard table $W(\Lambda,P)$ has a free pair.
\end{lemma}

\begin{proof}
Denote by $\ell_r$ the horizontal infinite half-line located to the right of  $V_r$, see Figure~\ref{fig:halfline}.
Let us consider all polygons $P+\rho$, $\rho\neq 0$ such that $\ell_r$ meets $P+\rho$ and denote by $Q(\rho)$ the first meeting point.
Finally choose $\rho\neq 0$ so that the distance between $V_r$ and $Q(\rho)$ is the smallest.
\begin{figure}[h]
\includegraphics[width=0.5\textwidth]{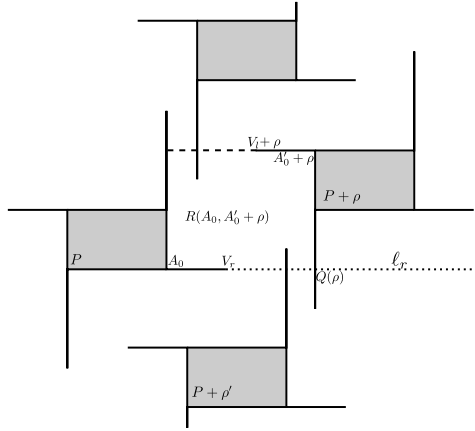}
\caption{Free pair in Case 1.}\label{fig:halfline}
\end{figure}

\textbf{Case 1.} Suppose that $Q(\rho)\neq V_l+\rho$, this is $\Im (V_l+\rho)\neq \Im V_r$. We will only show that $\Im (V_l+\rho)> \Im V_r$ implies $(A_0,A_0'+\rho)$ is free. A symmetric argument shows also that
$\Im (V_l+\rho)< \Im V_r$ yields $(A_1,A_1'+\rho)$ is free. Suppose, contrary to our claim, that $(A_0,A_0'+\rho)$ is not free and $P+\rho'$ intersects $R(A_0,A_0'+\rho)$. Since the height of $R(A_0,A_0'+\rho)$
is less than the height of $P$, $P+\rho'$ have to intersect the top or the bottom side of $R(A_0,A_0'+\rho)$. As the center of symmetry of $R(A_0,A_0'+\rho)$ is a center of symmetry of the lattice $\Lambda$,
it follows that there exists $P+\rho'$ intersecting the bottom side. The bottom side of $R(A_0,A_0'+\rho)$
is contained in the half-line $\ell_r$, so it follows that $Q(\rho')$ lies between $V_r$ and $Q(\rho)$. This contradicts the fact that the distance between $V_r$ and $Q(\rho)$ is  minimal.
\begin{figure}[h]
\includegraphics[width=0.5\textwidth]{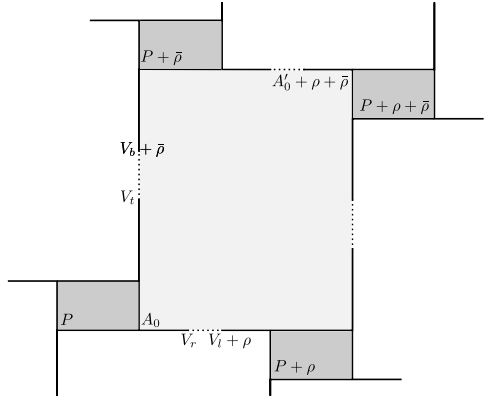}
\caption{Rectangle $R$.}\label{fig:shaded_rect}
\end{figure}

\textbf{Case 2.} Assume that $Q(\rho)=V_l+\rho$. Then the interval between $V_r$ and $V_l+\rho$ is horizontal and does not meet any polygon $P+\rho'$ for $\rho'\in\Lambda$.
We can also assume that there exists $\bar{\rho}\in\Lambda$ such that the interval between $V_t$ and $V_b+\bar{\rho}$ is vertical does not meet any polygon $P+\rho'$ for $\rho'\in\Lambda$. Otherwise,
we can proof the existence of a free pair using the vertical half-line located above the point $V_t$ and proceeding along the same lines as in Case 1.

Let $R:=R(A_0,A_0'+\rho+\bar{\rho})$, see Figure~\ref{fig:shaded_rect} (lightly shaded area).
Since the boundary of $R$ does not intersect polygons $P+\rho'$ for $\rho'\neq 0,\rho,\bar{\rho},\rho+\bar{\rho}$, for every $\rho'\in\Lambda$ we have
\[\rho'\in R\Leftrightarrow P+\rho'\subset R.\]

If $\Lambda\cap R=\emptyset$ then there is no $\rho'\in\Lambda$ such that $P+\rho'\subset R$. Then the pair $(A_0,A_0'+\rho+\bar{\rho})$ is free.

Next suppose that $\Lambda\cap R\neq\emptyset$. Choose $\widehat{\rho}\in \Lambda\cap R$ with the minimal distance to $0$. Let us consider the rectangle $R(A_0,A_0'+\widehat{\rho})$, see Figure~\ref{fig:shaded_rect2} (lightly shaded area).
\begin{figure}[h]
\includegraphics[width=0.5\textwidth]{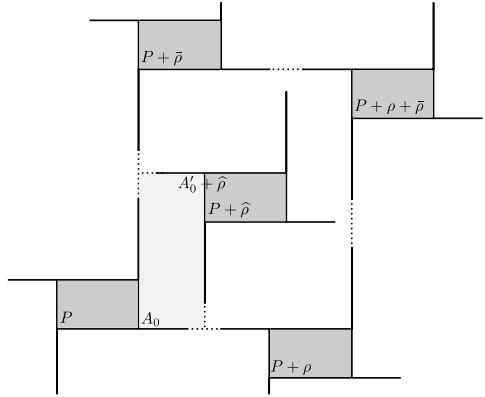}
\caption{Rectangle $R(A_0,A'_0+\rho)$.}\label{fig:shaded_rect2}
\end{figure}
Since $R(A_0,A_0'+\widehat{\rho})=R\cap(R+\widehat{\rho}-\rho-\bar{\rho})$, we have
\[\rho'\in R(A_0,A_0'+\widehat{\rho})\Leftrightarrow P+\rho'\subset R(A_0,A_0'+\widehat{\rho}).\]
By the minimality of the distance from $\widehat{\rho}$ to $0$, the set $R(A_0,A_0'+\widehat{\rho})\cap \Lambda$ is empty. Therefore,
$R(A_0,A_0'+\widehat{\rho})$ does not meet any polygon $P+\rho'$, so the pair $(A_0,A_0'+\widehat{\rho})$ is free, which completes the proof.
\end{proof}

\begin{lemma}\label{lem:typeC}
If $P$ is a rectangle with suckers of type b) or c) then for every $P$-admissible lattice $\Lambda$ the billiard table $W(\Lambda,P)$ has a free pair.
\end{lemma}

\begin{proof}
We will proceed only with the case c). The proof for the type b) runs similarly and we leave it to the reader.
 \begin{figure}[h]
\includegraphics[width=0.7\textwidth]{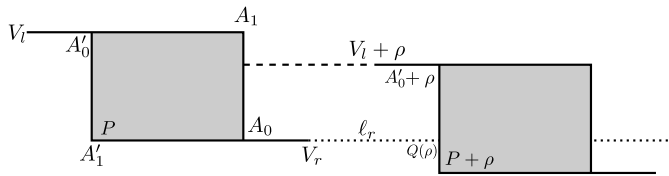}
\caption{Polygon $P$ in the case c).}\label{fig:typeC}
\end{figure}
Denote by $V_r,V_l,A_0,A'_0,A_1,A'_1$ the corners of $P$ as in Figure~\ref{fig:typeC}. Denote by $\ell_r$ the horizontal infinite half-line located to the right of  $V_r$.
Take $\rho\in\Lambda$ so that $P+\rho$ is the first scatterer that meets $\ell_r$ and let $Q(\rho)\in P+\rho$ be the point of the first meeting. If $Q(\rho)\neq V_l+\rho$ then
similar arguments to those applied in  the proof (Case 1) of Lemma~\ref{lem:typeA} show that the pair $(A_0,A_0'+\rho)$ is free, see Figure~\ref{fig:typeC}.

Suppose that $Q(\rho)=V_l+\rho$, see Figure~\ref{fig:typeC1}.
\begin{figure}[h]
\includegraphics[width=1\textwidth]{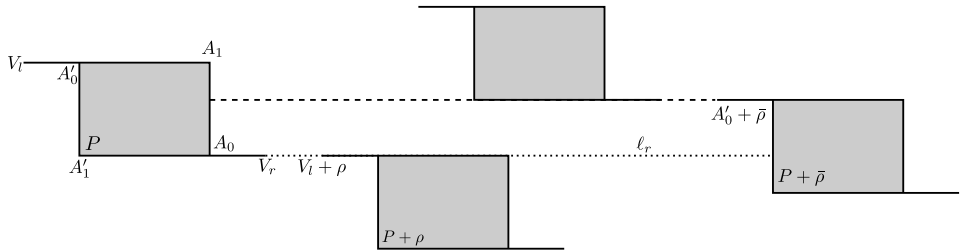}
\caption{The case $Q(\rho)=V_l+\rho$.}\label{fig:typeC1}
\end{figure}
Let $\bar{\rho}\in\Lambda$ so that $P+\bar{\rho}$ is the second scatterer intersecting $\ell_r$, see Figure~\ref{fig:typeC1}. Then
similar arguments to those applied in the proof (Case 1) of Lemma~\ref{lem:typeA} show also that the pair $(A_0,A_0'+\bar{\rho})$ is free, see Figure~\ref{fig:typeC1}.
This completes the proof for the type c).
\end{proof}

Finally, Lemmas~\ref{lem:typeA} and \ref{lem:typeA} combined with Remark~\ref{rem:rec} give the following conclusion.

\begin{corollary}
If $P$ is a rectangle with suckers of type a), b) or c) then for every $P$-admissible lattice $\Lambda$  a.e.\ direction on the billiard table $W(\Lambda,P)$ is recurrent.
\end{corollary}

\appendix
\section{Stable space and coboundaries.}\label{app}
In this section we deliver the proof of
Theorems~\ref{cohthm} and \ref{cohcor}.

Since the vertical flow for an Abelian
differential $r_{\pi/2-\theta}\omega\in \mathcal{M}_1(M)$
coincides with the directional flow in the direction $\theta$ for
the  Abelian differential $\omega\in \mathcal{M}_1(M)$, it
suffices to deal only with the vertical direction $\theta=\pi/2$. Therefore we will always assume that $0\in S^1$ is BOM generic.

\subsection{Preliminary definitions and notation}\label{notationappendix}
For any $\omega\in \mathcal{M}_1(M)$ let $M_{reg,\omega}$ be the
set of points which are regular both for the vertical and
horizontal flow on $(M,\omega)$ (that, we recall, means that both
flows are defined for all times).

For any $\omega\in \mathcal{M}_1(M)$ and any  point $p\in
M\setminus \Sigma$ let us denote by $I_{\omega}=I_{\omega}(p)$ the
horizontal interval on $(M, \omega)$ of total length $1$
centered at $p$.
\begin{remark}\label{nested}
Since the Teichm\"uller flow $(g_t)_{t \in \R}$ preserves
horizontal  leaves and rescales the horizontal vector fields by
$X_h^\omega=e^{t}X_h^{g_t\omega}$, we have that
\[
t< s \quad \Rightarrow \quad  I_{g_{s}\omega} (p)
\subset I_{g_{t}\omega} (p) .
\]
\end{remark}
In the rest of the Appendix we will consider $\omega$ and $p$ for which $I_{\omega}(p)$ satisfy:
\begin{align}\label{regularcondition}
\begin{aligned}
&\text{$I_{\omega} = I_{\omega}(p)$ has no self-intersections, does not
intersect $\Sigma$} \\ & \text{and all but finitely
many points from $I_{\omega}$ return to $I_{\omega}$
for the vertical flow.}
\end{aligned}
\end{align}
Then we will denote by   $T=T_{\omega}: I_{\omega} \to I_{\omega}$ the
Poincar{\'e} map of the vertical flow $(\varphi^v_t)_{t\in \R}$ on
$(M,\omega)$, which is well defined by (\ref{regularcondition})
and  is an IET. Let $\tau_{\omega}:I_{\omega} \to
\R^+$ the  first return time map and let us also denote by
$I_\alpha({\omega})$, $\alpha\in\mathcal{A}$ the subintervals exchanged by
$T_\omega$, by $ \lambda_\alpha({\omega})=\lambda_\alpha({\omega},p)$ their lengths and by
$\tau_\alpha({\omega})=\tau_\alpha({\omega},p)$  the first return time of the interval
$I_\alpha({\omega})$ to $I_{\omega}$.

Since $I_{\omega}(p)$  does not contain any singularity and the
set of singularities is discrete,  let  $\delta(\omega)=\delta({\omega},p)>0$ be  maximal
such that the strip
 \begin{equation*}
\bigcup_{0\leq t < \delta({\omega})} \varphi^v_t I_{\omega}(p)
 \end{equation*}
does not contain any singularities, and thus is isometric to the
Euclidean rectangle of height $\delta(\omega)$ and width $1$ in
the flat coordinates given by ${\omega}$.

For each $\alpha\in\mathcal{A}$ let $\xi_\alpha({\omega})=\xi_\alpha({\omega},p)\in H_1(M,\Z)$ be
the homology class obtained by considering the vertical trajectory
starting from any point from $I_\alpha(\omega)$ up to the first return to
$I_{\omega}$ and closing it up with a horizontal geodesic segment
contained in $I_{\omega}$.

\begin{remark}\label{remark:continuity}
For every pair $(\omega_0,p_0)\in
\mathcal{M}_1(M)\times(M\setminus\Sigma)$ satisfying
\eqref{regularcondition} there exists a sufficiently small
neighborhood $\mathcal{U}(\omega_0,p_0)\subset \mathcal{M}_1(M)$ of $\omega_0$
such that for any $\omega\in \mathcal{U}(\omega_0,p_0)$
\begin{itemize}
\item[(i)] the pair $(\omega,p_0$) also  satisfies
\eqref{regularcondition},
\item[(ii)]  the
induced IET $T_\omega$ on $I_\omega(p_0)$ has the same number
of exchanged intervals and the same combinatorial datum,
\item[(iii)] the quantities $\lambda_\alpha(\omega,p_0)$, $\tau_\alpha(\omega,p_0)$  for $\alpha\in\mathcal{A}$ and
$\delta(\omega, p_0)$ change continuously with $\omega\in
\mathcal{U}(\omega_0,p_0)$,
\item[(iv)]  the homology
class $\xi_\alpha(\omega,p_0)=\xi_\alpha(\omega_0,p_0)$ for every  $\alpha\in\mathcal{A}$ and $\omega\in
\mathcal{U}(\omega_0,p_0)$.
\end{itemize}
\end{remark}

\subsection{Auxiliary lemmas and proofs of main results}
Let $\omega\in\mathcal{M}_1(M)$ and let $0\in S^1$ is BOM generic.
Let $\mathcal{M}=\overline{SL(2,\R)\omega}$ and  let $\nu_{\mathcal{M}}$ be
the corresponding affine measure.
Choose  $\omega_0\in \mathcal{M}_1(M)$ in the support of the
measure $\nu_{\mathcal{M}}$ and let $p_0\in M\setminus \Sigma$ be such that its the vertical and the horizontal orbits are well defined.
Then the pair $(\omega_0,p_0)$ satisfies \eqref{regularcondition}.
Denote by $\mathcal{U}\subset \mathcal{M}_1(M)$ a
neighborhood of $\omega_0$ satisfying the claim of
Remark~\ref{remark:continuity}. Let
$\mathcal{T}\subset\mathcal{M}$ be an affine submanifold of
codimension $1$ containing $\omega_0$ and transversal to the
Teichm\"uller flow. Let $\mathcal{K}\subset \mathcal{T}$ be a closed ball  centered at $\omega_0$ and let $r>0$ so that
$\widetilde{\mathcal{K}}=\bigcup_{|t|\leq r}g_t\mathcal{K}\subset \mathcal{U}$.
Since $\omega_0$ belongs to the support of $\nu_{\mathcal{M}}$ and the measure
$\nu_{\mathcal{M}}$ is affine, we have $\nu_{\mathcal{M}}(\widetilde{\mathcal{K}})>0$ and
$\nu_{\mathcal{M}}(\partial\widetilde{\mathcal{K}})=0$.

\begin{lemma}\label{lem:Birk}
Let $(t_k)_{k\geq 0}$ be the sequence of successive positive
returns of $\omega$ to $\mathcal{K}$ for the Teichm\"uller flow.
This sequence is well defined, $t_k\to+\infty$ and $t_k/k\to
2r/\nu_{\mathcal{M}}(\widetilde{\mathcal{K}})$.
\end{lemma}

\begin{proof}
Since $0\in S^1$ is Birkhoff generic and
$\nu_{\mathcal{M}}(\partial\widetilde{\mathcal{K}})=0$, we have
\[\lim_{T\to\infty}\frac{1}{T}\int_0^T\chi_{\widetilde{\mathcal{K}}}(g_t\omega)\,dt=\nu_{\mathcal{M}}(\widetilde{\mathcal{K}}).
\]
As $\nu_{\mathcal{M}}(\widetilde{\mathcal{K}})>0$, this yields the claim of the lemma.
\end{proof}
The following lemma is a simple consequence of the fact that $\mathcal{K}$ is compact and ${\mathcal{K}}\subset \mathcal{U}$, see also \cite{Fr-Ulc:nonerg} for details.
\begin{lemma}\label{lemma:constructionK}
There exist positive
constants $A,C,c>0$ such that for every $\omega\in \mathcal{K}$
the pair $(\omega,p_0)$ satisfies \eqref{regularcondition},
\begin{equation}\label{baseuniform}
 \frac{1}{c} \| \gamma \|_\omega   \leq
\max_{\alpha\in\mathcal{A}} \left| \langle\xi_\alpha (\omega,p_0), \gamma \rangle\right|
\leq c\| \gamma \|_\omega \quad \text{for every} \quad
 \gamma \in H_1(M, \R),
\end{equation}
\begin{equation}\label{balance}
 \lambda_\alpha(\omega,p_0) \,
\delta(\omega,p_0)  \geq A \quad \text{and}\quad \frac{1}{C}\leq
\tau_\alpha(\omega,p_0)\leq C\quad\text{for any}\quad  \alpha\in\mathcal{A}.
\end{equation}
\end{lemma}

\begin{remark}\label{rem:Teich}
Recall that for any real $t$ the vertical and horizontal vector
fields  $X_v^\omega$ and $X_h^\omega$  on $(M, \omega)$ rescale as
follows under the Teichm\"uller geodesic flow
$(g_t)_{t\in\mathbb{R}}$:
\[X_v^\omega=e^{-t}X_v^{g_t\omega}, \qquad
X_h^\omega=e^{t}X_h^{g_t\omega}.
\]
Thus, the vertical and horizontal flows satisfy:
\[
\varphi^{v,\omega}_s p=\varphi^{v,g_t\omega}_{e^{-t}s}p, \qquad  \varphi^{h,\omega}_s p=\varphi^{h,g_t\omega}_{e^{t}s}p.\]
\end{remark}

\begin{notation}
For $0\leq t_0<t_1$, consider the intervals $I_{g_{t_0}\omega}$,
$I_{g_{t_1}\omega}$  defined at the beginning of the section,
that, by Remark \ref{nested}, satisfy $I_{g_{t_1}\omega}\subset
I_{g_{t_0}\omega}$ and for every regular point $p\in
I_{g_{t_0}\omega}$ denote respectively by
$\tau_{t_0,t_1}^{+}(p)\geq 0$ and $\tau_{t_0,t_1}^{-}(p)\leq 0$
the times of the first forward and respectively backward entrance
of the vertical orbit of $p$ to $I_{g_{t_1}\omega}$.
\end{notation}
\begin{lemma}\label{lem:twoentr}
Suppose that $0\leq t_0<t_1$ are such that $g_{t_0}\omega,g_{t_1}\omega\in\mathcal{K}$ and $I_{g_{t_0}\omega}$ is the shortest geodesic joining the ends of $I_{g_{t_0}\omega}$.
Then for every  $\gamma\in H_1(M,\R)$, $p\in I_{g_{t_0}\omega}$
and $\tau_{t_0,t_1}^{-}(p)\leq s\leq \tau_{t_0,t_1}^{+}(p)$ with
$\varphi^v_s(p)\in I_{g_{t_0}\omega}$ we have
\[\big|\langle \sigma^{v}_{s}(p),\gamma\rangle\big|\leq c\,C^2e^{t_1-t_0}\|\gamma\|_{g_{t_0}\omega}.\]
\end{lemma}

\begin{proof}
Let us assume that $0\leq s\leq\tau_{t_0,t_1}^{+}(p)$. The proof
for $\tau_{t_0,t_1}^{-}(p)\leq s<0$ is analogous. Denote by
$0=s_0<s_1<\ldots<s_K=s$ the consecutive return times (to
$I_{g_{t_0}\omega}$) of the forward vertical orbit of $p$. For
each pair  $s_{i-1}, s_{i}$ of consecutive return times of the
vertical flow $(\varphi^v_t)_{t\in\mathbb{R}}$ on $(M, \omega)$ to
the interval $I_{g_{t_0}\omega}$, it follows from Remark
\ref{rem:Teich} that $e^{-t_0}s_{i-1}, e^{-t_0}s_{i}$ are
consecutive return times of the vertical flow
$(\varphi^v_{g_{t_0}\omega})_{t\in\mathbb{R}}$ on $(M,
g_{t_0}\omega)$  to $I_{g_{t_0}\omega}$. Thus, since the first
return time function of
$(\varphi^v_{g_{t_0}\omega})_{t\in\mathbb{R}}$  to
$I_{g_{t_0}\omega}$ assumes the finitely many values
$\tau_\alpha(g_{t_0}\omega)$, $\alpha\in\mathcal{A}$,  for all $0\leq i <K $ we have
\[
e^{-t_0}s_{i}-e^{-t_0}s_{i-1}\geq \min_{\alpha\in\mathcal{A}}\tau_\alpha(g_{t_0}\omega).
\]
It follows that
\[s=s_K\geq
Ke^{t_0}\min_{\alpha\in\mathcal{A}}\tau_\alpha(g_{t_0}\omega).\]
Moreover, the
orbit segment
\[\{\varphi^{v,\omega}_tp:s_0<t<s_K\}=\{\varphi^{v,g_{t_1}\omega}_tp: 0<t<e^{-t_1}s_K\}\]
does not intersect the interval $I_{g_{t_1}\omega}$. It follows
that \[e^{-t_1}s_{K} \leq \max_{\alpha\in\mathcal{A}}\tau_\alpha(g_{t_1}\omega).\] Therefore,
\[K\leq \frac{e^{t_1}\max_{\alpha\in\mathcal{A}}\tau_\alpha(g_{t_1}\omega)}{e^{t_0}\min_{\alpha\in\mathcal{A}}\tau_\alpha(g_{t_0}\omega)}.\] In view of \eqref{balance}, it follows
that $K\leq e^{t_1-t_0}C^2$ and $s\leq e^{t_1}C$.
Next note that
\[\sigma^v_{s}(p)=\sum_{i=1}^K\sigma^v_{s_i-s_{i-1}}(\varphi_{s_{i-1}}^vp)\]
and every class $\sigma^v_{s_i-s_{i-1}}(\varphi_{s_{i-1}}^vp)$ is
equal to some $\xi_\alpha(g_{t_0}\omega)\in H_1(M,\Z)$. Thus, by Lemma
\ref{lemma:constructionK} applied to
$g_{t_0}\omega\in\mathcal{K}$, we have
\[\big|\langle\sigma^v_{s}(p),\gamma\rangle\big|\leq
\sum_{i=1}^K\big|\langle\sigma^v_{s_i-s_{i-1}}(\varphi_{s_{i-1}}^vp),\gamma\rangle\big|\leq Kc\|\gamma\|_{g_{t_0}\omega}.\]
It follows that
$\big|\langle\sigma^v_{s}(p),\gamma\rangle\big|\leq
cC^2e^{t_1-t_0}\|\gamma\|_{g_{t_0}\omega}$.
\end{proof}

\begin{proof}[Proof of {Theorem}~\ref{cohthm}]
By Lemma~\ref{lem:Birk}, there exists an increasing sequence
$(t_k)_{k\geq 0}$ of positive numbers such that
$g_{t_k}\omega\in\mathcal{K}$ for $k\geq 0$ and $t_k/k\to 2r/\nu_{\mathcal{M}}(\mathcal{K})>0$.

Let us consider the
sequence of intervals $(I_{g_{t_k}\omega})_{k\geq 0}$ centered
at $p_0$. By Remark~\ref{nested}, $(I_{g_{t_k}\omega})_{k\geq
0}$ is a decreasing sequence of nested intervals. Without loss of generality we can assume that
$I_{g_{t_0}\omega}$ is the shortest geodesic joining the ends of $I_{g_{t_0}\omega}$.
For every pair of regular points $p_1,p_2$ in $(M,\omega)$ denote by $\sigma(p_1,p_2)\in H_1(M,\Z)$
the homology class of the closed oriented curve compose form the forward segment of the vertical orbit from $p_1$ to the first hit of $I_{g_{t_0}\omega}$,
the backward segment of the vertical orbit from $p_2$ to its first hit of $I_{g_{t_0}\omega}$ closed by a directed segment of the interval $I_{g_{t_0}\omega}$
and the shortest path from $p_2$ to $p_1$ in $(M,\omega)$. Then
\[C':=\sup\{\|\sigma(p_1,p_2)\|_\omega:p_1,p_2\}\ \text{ is finite.}\]

Fix a regular
point  $p$ in $(M,\omega)$.  For any $t >0$, the trajectory
$\Phi_t:=\{\varphi^v_sp:0\leq s\leq t\}$ can be inductively
decomposed into principal return trajectories as follows
(analogously to Lemma 9.4 in \cite{For-dev}). Let $K \in \N$  be
the maximum $k \geq 0$ such that $\Phi_t$ intersect
$I_{g_{t_k}\omega}$.  For every $k=0,\ldots,K$ let $0\leq l_k\leq
r_k\leq t$ be the times of the first and the last intersection of
$\Phi_t$ with $I_{g_{t_k}\omega}$. Then, since the intervals
$(I_{g_{t_k}\omega})_{k\geq 0}$ are nested,
\[0\leq l_0\leq l_1\leq\ldots\leq l_K\leq r_K\leq\ldots\leq r_1\leq r_0\leq t.\]
Then
\[\sigma^v_t(p)=\sigma(p,\varphi^v_{t}p)+\sum_{i=1}^K\sigma^v_{l_i-l_{i-1}}(\varphi^v_{l_{i-1}}p)+\sigma^v_{r_K-l_{K}}(\varphi^v_{l_{K}}p)+
\sum_{i=1}^K\sigma^v_{r_i-r_{i-1}}(\varphi^v_{r_{i-1}}p).\]
Moreover, for $i=1,\ldots,K$ we have
\[
l_{i}-l_{i-1}=\tau_{t_{i-1},t_i}^+(\varphi^v_{l_{i-1}}p), \; r_{i}-r_{i-1}=\tau_{t_{i-1},t_i}^-(\varphi^v_{r_{i-1}}p)\text{  and }
r_K-l_K\leq \tau_{t_{K},t_{K+1}}^+(\varphi^v_{l_K}p),
\]
where the functions $\tau_{t_{i-1},t_i}^{\pm}$ are defined before Lemma~\ref{lem:twoentr}.
In view of
Lemma~\ref{lem:twoentr}, it follows that for every $1\leq i\leq K$ and $\gamma\in H_1(M,\R)$ we have
\begin{equation}\label{eq:sum1}
\big|\langle \sigma^v_{l_i-l_{i-1}}(\varphi^v_{l_{i-1}}p),\gamma\rangle\big|\leq
c\,C^2e^{t_i-t_{i-1}}\|\gamma\|_{g_{t_{i-1}}\omega},
\end{equation}
\begin{equation}\label{eq:sum2}
 \big|\langle \sigma^v_{r_i-r_{i-1}}(\varphi^v_{r_{i-1}}p),\gamma\rangle\big|\leq
c\,C^2e^{t_i-t_{i-1}}\|\gamma\|_{g_{t_{i-1}}\omega}
\end{equation}
and
\begin{equation}\label{eq:sum3}
 \big|\langle \sigma^v_{r_K-l_K}(\varphi^v_{l_{K}}p),\gamma\rangle\big|\leq
c\,C^2e^{t_{K+1}-t_{K}}\|\gamma\|_{g_{t_{K}}\omega}.
\end{equation}
Moreover,
\begin{equation}\label{eq:sum4}
 \big|\langle \sigma(p,\varphi^v_{t}p),\gamma\rangle\big|\leq \|\sigma(p,\varphi^v_{t}p)\|_{\omega}\|\gamma\|_{\omega}\leq
C'\|\gamma\|_{\omega}.
\end{equation}
Summing \eqref{eq:sum1}-\eqref{eq:sum4} we get
\begin{equation}\label{eq:sumsum}\big|\langle \sigma^v_{t}(p),\gamma\rangle\big|\leq
2\sum_{k=0}^\infty c\,C^2e^{t_{k+1}-t_k}\|\gamma\|_{g_{t_k}\omega}+ C'\|\gamma\|_{\omega}.
\end{equation}

Assume that $\gamma\in E_\omega^-$. Then
there exist constants $C_1,
\theta >0$ such that $\|\gamma\|_{g_{t_k}\omega} \leq C_1
e^{-\theta t_k}$ for all $k\geq 0$. Using this inequality together
with (\ref{eq:sumsum}), we get that there exists $C_2>0$ such that
for any  $t\geq 0$, one has
\begin{equation}\label{decompintegrals1}
\big|\langle \sigma^v_{t}(p),\gamma\rangle\big| \leq  C_2
\sum_{k=0}^\infty e^{( t_{k+1}-t_k)}  e^{-\theta t_k} +C_2 = C_2
\sum_{k=0}^\infty  e^{ \left( \frac{ t_{k+1}-t_k}{t_k}  - \theta
\right)t_k } + C_2 .
\end{equation}
Since  $t_k /k\to 2r/\nu_{\mathcal{M}}(\mathcal{K})>0$, we have $(t_{k+1}-t_k)/t_k\to 0$. Thus, $(t_{k+1}-t_k)/t_k - \theta  \leq - \theta/2$ for all $k$
sufficiently large. It follows that the above series is convergent and $|\langle \sigma^v_{t}(p),\gamma\rangle|$ is uniformly bounded for all  $t\geq 0$
and regular points $p$ in $(M,\omega)$. This concludes the proof of the first part of Theorem~\ref{cohthm}.

\smallskip
Suppose that $H_1(M,\Q)=K\oplus K^{\perp}$ is an orthogonal symplectic splitting  such that $V=\R\otimes_{Q}K$ and let
$(\xi_i)_{i=1}^{2d}$ be a basis of $K$.  After some rescaling we can assume that
$\xi_i\in H_1(M,\Z)$ for $1\leq i\leq 2d$.   We now show that for every nonzero $\gamma\in E^-_{\omega}$ we have
$(\langle\gamma,\xi_i\rangle)_{i=1}^{2d}\notin \R \cdot \Q^{2d}$. Suppose, contrary to our claim, that
there exist $q\in\N$ and $a>0$ such that $\langle\gamma,\xi_i\rangle\in a\Z/q$ for $1\leq i\leq 2d$.
Since $g_{t_k}\omega\in\mathcal{K}$ for $k\geq 0$, in view of \eqref{baseuniform}, we have
\begin{equation}\label{eq:stab}
\frac{1}{c} \| \gamma \|_{g_{t_k}\omega}   \leq
\max_{\alpha\in\mathcal{A}} \left| \langle\xi_\alpha (g_{t_k}\omega), \gamma \rangle\right|
\leq c\| \gamma \|_{g_{t_k}\omega} \quad \text{for} \quad
k\geq 0.
\end{equation}
Denote by $p:H_1(M,\Q)\to K$ the orthogonal symplectic projection. Then
\[p(\varrho)=\sum_{j=1}^{2d}\sum_{i=1}^{2d}\langle\varrho,\xi_i\rangle\,\Xi^{-1}_{ij}\xi_j,\]
where $\Xi=[\langle\xi_i,\xi_j\rangle]_{i,j=1,\ldots,2d}$. Let
$Q:=\det\Xi\in\Z\setminus\{0\}$. Since $\gamma\in V$ and $V$ is
orthogonal to $K^\perp$, for every $\varrho\in H_1(M,\Z)$ we have
\[\langle\varrho,\gamma\rangle=\langle p(\varrho),\gamma\rangle=
\sum_{j=1}^{2d}\sum_{i=1}^{2d}\langle\varrho,\xi_i\rangle\,\Xi^{-1}_{ij}\langle\xi_j,\gamma\rangle\in \frac{a}{qQ}\Z.\]
Therefore  $\max_{\alpha\in\mathcal{A}} \left| \langle\xi_\alpha
(g_{t_k}\omega), \gamma \rangle\right|\in a\Z/qQ$ for $k\geq 0$.
As $\gamma\in E^-_\omega$, we have $\| \gamma
\|_{g_{t_k}\omega}\to 0$ as $k\to+\infty$. By \eqref{eq:stab}, it
follows that $\max_{\alpha\in\mathcal{A}}| \langle\xi_\alpha
(g_{t_k}\omega), \gamma \rangle|=0$ for all $k$ large enough.
Again, by \eqref{eq:stab}, this gives $\| \gamma
\|_{g_{t_k}\omega}=0$ for all $k$ large enough and hence
$\gamma=0$, contrary to assumption.
 This concludes the proof of the second part of Theorem
\ref{cohthm}.\end{proof}

\begin{proof}[Proof of Theorem~\ref{cohcor}]
Let $I:=I_{g_{t_0}\omega}(p_0)$.  Take any regular point $x\in I$. For
every $n\in \N$ denote by $s_n>0$ the time of $n$-th return of $x$
to $I$ for the flow $(\varphi^\theta_t)_{t\in\R}$. Then
$\psi_\gamma^{(n)}(x)=\langle\gamma,\sigma^v_{s_n}(x)\rangle$.
In view of Theorem~\ref{cohthm}, $|\psi_\gamma^{(n)}(x)|\leq C$
for  every regular $x\in I$ and $n\geq 1$ which shows that
$\psi_\gamma:I\to \R$ is a coboundary.
\end{proof}

\end{document}